\documentclass[reqno,12pt]{amsart}
\usepackage[utf8]{inputenc}
\usepackage[left=1.2in, right=1.2in, top=1in]{geometry}

\usepackage{esint}
\usepackage{amsmath,amssymb,amsthm,mathrsfs,color,times,textcomp,verbatim,mathtools}
\usepackage{xcolor}
\usepackage[colorlinks=true]{hyperref}
\hypersetup{urlcolor=blue, citecolor=red, linkcolor=blue}
\usepackage[square,numbers]{natbib}

\numberwithin{equation}{section}
\theoremstyle{plain}
\newtheorem{theorem}{Theorem}[section]
\newtheorem{proposition}[theorem]{Proposition}
\newtheorem{lemma}[theorem]{Lemma}

\newtheorem{claim}{Claim}

\usepackage{graphicx}

\newtheorem{definition}[theorem]{Definition}

\newtheorem{remark}[theorem]{Remark}

\title[Non-degeneracy and stability of half-harmonic maps]{Non-degeneracy and quantitative stability of half-harmonic maps from $\R$ to $\S$}
\author{Bin Deng}
\address{School of Mathematical Sciences, University of Science and Technology of China,
Hefei, Anhui Province, P.R. China, 230026}
\email{bingomat@mail.ustc.edu.cn}
\author{Liming Sun}
\address{Academy of Mathematics and Systems Science, the Chinese Academy of Sciences, Beijing 100190, China.}
\email{lmsun@amss.ac.cn}
\author{Jun-cheng Wei}
\address{Department of Mathematics, University of British Columbia, Vancouver, BC, V6T 1Z2, CA.}
\email{jcwei@math.ubc.ca}
\date{\today \,(Last Typeset)}

\def\R{\mathbb{R}}
\def\S{\mathbb{S}}
\def\Rp{\mathbb{R}^2_+}
\def\D{\mathbb{D}}
\def\deg{\text{deg\,}}
\def\cn{\dot H^{1/2}(\S)}
\def\cnr{\dot H^{1/2}(\R)}
\def\cnS{\dot H^{1/2}(\S;\S)}
\def\cnRS{\dot H^{1/2}(\R;\S)}
\def\id{\text{id}_{\S}}
\def\DRh{(-\Delta_{\R})^{\frac12}}
\def\DSh{(-\Delta_{\S})^{\frac12}}
\def\Z{\mathbb{Z}}
\def\Im{\textup{Im\,}}
\def\Re{\textup{Re\,}}
\def\im{\textup{i\,}}
\def\u{\boldsymbol{u}}
\def\v{\boldsymbol{v}}
\def\l{\lambda}
\def\bphi{\boldsymbol{\phi}}
\def\bpsi{\boldsymbol{\psi}}
\def\K{\boldsymbol{K}_{\vartheta,\vec{\alpha}}}

\begin{document}

\maketitle
\begin{abstract}
We consider half-harmonic maps from $\mathbb{R}$ (or $\mathbb{S}$) to $\mathbb{S}$. We prove that all (finite energy) half-harmonic maps are non-degenerate. In other words, they are integrable critical points of the energy functional. A full description of the kernel of the linearized operator around each half-harmonic map is given. The second part of this paper devotes to studying the quantitative stability of half-harmonic maps. When its degree is $\pm 1$, we prove that the deviation of any map $\boldsymbol{u}:\mathbb{R}\to \mathbb{S}$ from M\"obius transformations can be controlled uniformly by $\|\boldsymbol{u}\|_{\dot H^{1/2}(\mathbb{R})}^2-\deg \boldsymbol{u}$. This result resembles the quantitative rigidity estimate of degree $\pm 1$ harmonic maps $\mathbb{R}^2\to \mathbb{S}^2$ which is proved recently.  Furthermore, we address the quantitative stability for half-harmonic maps of higher degree.  We prove that if $\boldsymbol{u}$ is already near to a Blaschke product, then the deviation of $\boldsymbol{u}$ to Blaschke products can be controlled by $\|\boldsymbol{u}\|_{\dot H^{1/2}(\mathbb{R})}^2-\deg \boldsymbol{u}$. Additionally, a striking example is given to show that such quantitative estimate can not be true uniformly for all $\boldsymbol{u}$ of degree 2. We conjecture similar things happen for harmonic maps $\R^2\to \S^2$.
\end{abstract}

\setlength{\parskip}{0.5em}
\section{Introduction}
\subsection{Motivation and main results}
The analysis of critical points of conformal invariant lagrangians has drawn much attention since 1950, due to their important applications in physics and geometry. One of the prominent examples is harmonic maps $\u:\Omega\to \mathbb{S}^n$, which are critical points of the Dirichlet energy
\begin{align}
    \mathcal{E}(\u)=\int_{\Omega}|\nabla \u|^2dx.
\end{align}
When the domain $\Omega$ is a subset of $\R^2$, $\mathcal{E}(\u)$ is conformally invariant and this plays a crucial role in the regularity theory of such maps (see \citet{helein1990regularite}, \citet{riviere2007conservation} and references therein). The theory has been generalized to even-dimensional domains whose critical are called poly-harmonic maps.

In the recent years, many authors are interested in the analog of Dirichlet energy in odd-dimensional
cases, for instance, \citet{DaLio2013,da2015compactness}, \citet{DaLioRiv2011-1,DaLioRiv2011-2}, \citet{millot2015fractional},
\citet{schikorra2012regularity} and the references therein. In these works, a special but quite interesting case is the so-called half-harmonic maps from $\R$ into $\S$ which can be defined as critical points of the following line energy
\begin{align}\label{intro:E(u)}
    \mathcal{E}(\u)=\frac{1}{2\pi}\int_{\R}|(-\Delta_{\R})^{\frac14}\u|^2dx.
\end{align}
The functional $\mathcal{E}$ enjoys  invariance under the M\"obius group which is the trace of conformal maps keeping invariant the half-space $\Rp$. In fact, $\mathcal{E}(\u)$ coincides with $\|\u\|_{\dot H^{1/2}(\R)}^2$, see \eqref{1.1} below. Computing the associated Euler-Lagrange equation for \eqref{intro:E(u)}, it is easy to see that if $\u:\R\to \S$ is a half-harmonic map, then $\u$ satisfies
\begin{align}\label{intro:E-L}
    (-\Delta)^{\frac{1}{2}} \u(x)=\left(\frac{1}{2 \pi} P.V.\int_{\mathbb{R}} \frac{|\u(x)-\u(s)|^{2}}{|x-s|^{2}} d s\right) \u(x) \quad\text { in } \mathbb{R}.
\end{align}
Fundamental regularity of half-harmonic maps has been obtained in \cite{DaLioRiv2011-1,DaLioRiv2011-2}. A complete classification has been known by \cite{millot2015fractional,da2015compactness} (cf. Theorem \ref{thm:class} below). Associating\footnote{Throughout this paper, bold font $\u$ denotes a vector-valued map, while $u$ means a complex-valued map.} $\u=(u_1,u_2)$ to a complex function $u=u_1+\im u_2$, all half-harmonic maps consist of the following products and their complex conjugates
\begin{align}\label{intro:1/2-H-form}
\psi_{\vartheta,\vec{\alpha}}:=e^{\im \vartheta} \prod_{k=1}^{d} \frac{x-\alpha_k}{x-\bar\alpha_k}
\end{align}
where $\vartheta\in \S$, $\vec{\alpha}=(\alpha_1,\cdots,\alpha_d)$ with
$\alpha_k=x_k+\im \l_k$ lies in the upper half plane $\mathbb{H}$, and $d=\deg \u$ (see \eqref{deg-series} for its definition). Modulo a Cayley transformation, the above expressions are equivalent to Blaschke products or their conjugates.

Apart from the strong analogy to harmonic maps on $\R^2$, half-harmonic maps have intricate connections to minimal surfaces with free boundary, for instance see \cite{fraser2016sharp,laurain2017regularity,millot2015fractional,jost2019qualitative}. On the other hand, in recent years, several papers were devoted to the study of the fractional Sobolev space $H^{1 / 2}$ with values into the circle, in particular in the framework of the Ginzburg-Landau model, see the paper \citet{mironescu2004variational} and reference therein. The simplest of such spaces is $H^{1 / 2}\left(\R ; \S\right)$.



The main purpose of this paper is twofold: First, we prove that each (finite energy)\footnote{According to \cite{enno2018energy}, we do have some maps $\u:\R\to \S$ satisfies \eqref{intro:E-L} but has infinite line energy. We do not study it here.} half-harmonic map is non-degenerate by characterising the kernel of the linearized operator around each half-harmonic map.
Second, we study the quantitative stability estimates near each half-harmonic map. The non-degeneracy and stability are crucial to the half-harmonic map heat flow, which is an analogy of harmonic map heat flow. There are vibrant researches along this direction by \citet{sire2017infinite}, \citet{Sch-Sire-Wang}, \citet{wettstein2021uniqueness}. An interesting conjecture in \cite{sire2017infinite} states that   half-harmonic map heat flow only blow-ups in infinite time, which is quite different from what we know about harmonic map heat flow.

Evidently, differentiating \eqref{intro:1/2-H-form} with $\vartheta$, $\Re \alpha_k$, $\Im \alpha_k$, for $k=1,\cdots,d$ respectively, generates kernel maps for the linearized operators $L_{\psi_{\vartheta,\vec{\alpha}}}$ as
\begin{align}
\begin{split}
    L_{\psi_{\vartheta,\vec{\alpha}}}(\v)=\DRh \v(x)&-\left(\frac{1}{2 \pi} P.V.\int_{\mathbb{R}} \frac{|\bpsi_{\vartheta,\vec{\alpha}}(x)-\bpsi_{\vartheta,\vec{\alpha}}(s)|^{2}}{|x-s|^{2}} d s\right)\v(x)\\
    &-\left(\frac{1}{\pi} \int_{\mathbb{R}} \frac{(\bpsi_{\vartheta,\vec{\alpha}}(x)-\bpsi_{\vartheta,\vec{\alpha}}(y)) \cdot(\v(x)-\v(y))}{|x-y|^{2}} d y\right) \u(x)
\end{split}
\end{align}
where $\v\in \dot H^{1/2}(\R;\S)$ with $\bpsi_{\vartheta,\vec{\alpha}}\cdot\v=0$. Conversely, if the kernels maps of $L_{\psi_{\vartheta,\vec{\alpha}}}$ are all generated by differentiating nearby half-harmonic maps, then we call $\bpsi_{\vartheta,\vec{\alpha}}$ is non-degenerate. Such property is also known as integrability in the context of minimal surfaces \cite{allard1981radial} and harmonic maps \cite{gulliver1989rate}. The non-degeneracy of harmonic maps from $\R^2$ to $\S^2$ has been established in \cite{gulliver1989rate,chen2018nondegeneracy}.

The study to non-degeneracy of half-harmonic maps is initiated by \citet{sire2017nondegeneracy} and \citet{enno2018energy}.
The authors in \cite{sire2017nondegeneracy} confirm the case when $\deg \u=\pm 1$. The authors in  \cite{enno2018energy} can deal with very special case of higher degree, more precisely, when $u=(x-\im)^m/(x+\im)^m$. After a stereographic projection (or Cayley transformation) such $u$ is equivalent to $z^m$, $z\in\S$. Their approach depend on the symmetry of $z^m$ crucially. In the present paper, we completely solve the non-degeneracy for all half-harmonic maps.
\begin{theorem}\label{intro:thm:non}
Suppose $u$ is a half-harmonic map $\R\to \S$, then all the $\dot H^{1/2}(\R)$ maps in the kernel of $ L_{u}$ are generated by differentiating half-harmonic maps close to $u$. More precisely, $\text{dim}_\R\ker L_{\u}=2| \deg u |+1$. In particular, if $u=e^{\im \vartheta}\prod_{j=1}^k(\frac{x-\alpha_j}{x-\bar \alpha_j})^{d_j}$ with $\{\alpha_j\}_{j=1}^k$ are distinct and $d_j\geq 1$, then
\begin{align*}
    \ker L_{u}=\textup{span}_{\R}\left\{1,\Re \frac{1}{(x-\bar \alpha_j)^{s}},\Im \frac{1}{(x-\bar\alpha_j)^{s}}: s=1,\cdots,d_j,j=1,\cdots,k.\right\}\im u.
\end{align*}
\end{theorem}

The second part of this paper deals with the quantitative stability of half-harmonic maps. To describe it, we note that half-harmonic maps achieve the minimum of $\dot H^{1/2}(\R;\S)$ norm in its homotopy class.  Namely
\begin{theorem}\label{intro:lower-bd}
Suppose that $u\in\cnRS$. Then $\|u\|_{\cnr}^2\geq |\deg u|$. If the equality holds, then $u$, or its complex conjugate, must be the form of \eqref{intro:1/2-H-form}.
\end{theorem}
A natural question is that whether the discrepancy $\|u\|_{\cnr}^2-|\deg u|$ can control quantitatively the difference of $u$ from the half-harmonic maps. Naively, one expects a linear control as
\begin{align}\label{stab-1/2}
    \inf_{\vartheta\in \S,\vec{\alpha}\in \mathbb{H}^d }\|u-\psi_{\vartheta,\vec{\alpha}}\|_{\cnr}^2\leq C\left(\|u\|_{\cnr}^2-|\deg u|\right)
\end{align}
where $C$ is independent of $u$.

Such type of question has been raised for many other topics. For instance, \citet{brezis1985sobolev} ask a similar question of the classical Sobolev inequality on $\R^n$. Later \citet{bianchi1991note} obtain a quantitative stability estimate in the spirit of  \eqref{stab-1/2}. \citet{fusco2008sharp} prove a  sharp quantitative stability about isoperimetric inequality. To authors' knowledge, other types of quantitative stability estimates include (but not limited to) \cite{figalli2020sharp,karpukhin2021stability,allen2021linear}. Reader can see their papers and reference therein.

Recently \citet{bernand2021quantitative} prove a quantitative stability for degree $\pm 1$ harmonic maps from $\R^2\to \S^2$ similar to \eqref{stab-1/2}, whose proof is simplified by \cite{hirsch2021note,topping2020rigidity}. Due to the strong analogy between harmonic maps and half-harmonic ones, their works inspire us to prove the following theorem.  Denote
\begin{align}\label{def:D(u)}
    D(u)=\|u\|_{\cnr}^2-|\deg u|.
\end{align}
\begin{theorem}\label{intro:thm:sta-type1}
For $u\in\dot{H}^{1/2}(\R;\S)$ with $\deg u= 1$, there exists an $\alpha\in \mathbb{H}$ and $\vartheta\in\S$ such that
\begin{align}\label{inf_d-D}
    \|u-\psi_{\vartheta,\alpha}\|_{\cnr}^2\leq 36D(u).
\end{align}
If $\deg u=-1$, then the above statement holds with $\bar u$.
\end{theorem}

The above theorem leaves us an intriguing question for half-harmonic maps with higher degree. The answer to this could shed some light to harmonic maps with higher degree. We find that there are some fundamental differences between the case degree $\pm1$ and higher degree. For instance, in degree 1, $\psi_{\vartheta,\alpha}$ is equivalent to $\psi_{0,\im}$ after some harmless rotation and M\"obius transformation of $\mathbb{H}$, while both sides of  \eqref{inf_d-D} is invariant under these operations. Essentially, the stability estimates near $\psi_{\vartheta,\alpha}$ is equivalent to that of $\psi_{0,\im}$. Thus we have a uniform constant in \eqref{inf_d-D}. However, in higher degree we do not have such equivalence. For instance, consider $\psi_{\vartheta,(\alpha_1,\alpha_2)}$ in degree 2. One can use M\"obius transformation to bring $\alpha_1$ to $\im$, but there is no control of $\alpha_2$, which might be very near to the boundary of $\mathbb{H}$. Indeed, we prove the following dichotomy for $\deg=2$.
\begin{theorem}\label{intro:thm:un-stable}
For any $M>0$, there exists $u\in \dot H^{1/2}(\R;\S)$ with $\deg u=2$ such that
\begin{align}\label{intro:inf>M}
    \inf_{\vartheta\in \S,\vec{\alpha}\in \mathbb{H}^2 }\|u-\psi_{\vartheta,\vec{\alpha}}\|_{\dot H^{1/2}(\R)}^2\geq M \left(\|u\|_{\dot H^{1/2}(\R)}^2-2\right).
\end{align}
\end{theorem}
The function $u$ we choose is a perturbation of $\psi_{0,\vec{\alpha}}$ with $\alpha_1=j^2+\im$ and $\alpha_2=-j^2+\im$ with $j\to \infty$. This exactly captures the dilemma in higher degree mentioned above. This example shows that one should not hope the stability for higher degree as simple as \eqref{stab-1/2}. Nevertheless, we can prove a local version of quantitative stability. Namely, if $u$ is already sufficiently close  to some half-harmonic map, then \eqref{stab-1/2} still holds.
\begin{theorem}\label{intro:thm:local}
For any compact set $\Omega\Subset \mathbb{H}$, there exists $\delta_{d,\Omega,\varepsilon}$ with the following significance. Suppose $u\in\dot H^{1/2}(\R;\S)$ satisfies $\deg u=d>0$ and
\begin{align}\label{intro:u-phi<delta}
    \|u-\psi_{\vartheta,\vec{\beta}}\|_{\cnr}^2\leq \delta_{d,\Omega,\varepsilon}
\end{align}
for some $\vartheta\in \S$ and $\vec{\beta}\in \Omega^d$. Then there exists a constant $C_{d,\Omega,\varepsilon}>0$ (independent of $u$), $\vartheta'\in \S$ and $\vec{\alpha}\in \Omega_\varepsilon^d \Subset\mathbb{H}^{d}$ such that
\begin{align}\label{intro:aDu}
    \|u-\psi_{\vartheta',\vec{\alpha}}\|_{\cnr}^2\leq C_{d,\Omega,\varepsilon} D(u).
\end{align}
Here $\Omega_{\varepsilon}=\cup_{p\in \Omega}B_{\varepsilon}(p)$ with $B_{\varepsilon}(p)$ are open balls in $\D$.
\end{theorem}

For harmonic maps, one can expect similar things happening here. We conjecture that there is local version of quantitative stability, while no uniform one as simple as \eqref{stab-1/2} holds for higher degree.


\subsection{Comments on the proofs}
We briefly sketch the main idea behind the proofs of Theorem \ref{intro:thm:non}-\ref{intro:thm:local}. To show the non-degeneracy, we start with the harmonic extension (denote as $\boldsymbol{U}$) of any half-harmonic map $\u$ to $\Rp$.  It has been proven by \cite{millot2015fractional,enno2018energy} that the Hopf differential of $\boldsymbol{U}$
\begin{align}\label{intro:Phi1}
\Phi=(|\partial_x\boldsymbol{U}|^2-|\partial_y\boldsymbol{U}|^2)-2\im \partial_x\boldsymbol{U}\cdot \partial_y\boldsymbol{U}\equiv 0.
\end{align}
This actually implies $U$ must be a holomorphic or anti-holomorphic function on $\mathbb{H}$.
If $v\in \ker L_u$, then we anticipate its harmonic extension $V$ is also holomorphic or anti-holomorphic on $\mathbb{H}$. This is done by considering
\begin{align}
    \tilde \Phi=2U_z\bar V_z=\partial_x \boldsymbol{U}\cdot \partial_x \boldsymbol{V}-\partial_y \boldsymbol{U}\cdot \partial_y \boldsymbol{V}-\im(\partial_x \boldsymbol{U}\cdot \partial_y \boldsymbol{V}+\partial_x \boldsymbol{V}\cdot \partial_y \boldsymbol{U}).
\end{align}
Note that $\tilde \Phi$ can be thought of the derivative of $\Phi$ in \eqref{intro:Phi1}. One can show that $\tilde \Phi\equiv 0$ which implies $V$ is (anti-)holomorphic on $\mathbb{H}$ when $U$ is (anti-)holomorphic. A crucial step is defining $W=V/U$. Since $\u\cdot\v=0$ on $\R$, then $W$ is purely imaginary on $\partial \mathbb{H}$. By Scharwz reflection, we can extend $W$ to a meromorphic function on $\mathbb{C}$. One can show that $W$ has no essential pole at infinity and thus $W$ is a rational functions. By counting the dimension of admissible rational functions, we get $\dim_{\R}\ker L_u=2|\deg u|+1$. This exactly matches the dimension of parameters generating nearby half-harmonic maps.

The proof of Theorem \ref{intro:thm:sta-type1} follows closely the approach in \cite{hirsch2021note}. They use harmonic polynomials to do the computation there, while here we use Fourier series of functions $\S\to \S$ instead. Thanks to the fact that the $\cn$ norm and degree of any map $u\in \dot H^{1/2}(\S;\S)$ can be written explicitly using the coefficients of Fourier series, the proof here is much shorter than that in \cite{hirsch2021note}. The proof of Theorem \ref{intro:thm:local} is a carefully refinement of the case degree $\pm 1$ and induction. Non-degeneracy can be used to prove quantitative stability estimates in some cases. The reason behind this is that the linearized operator has a spectral gap on the orthogonal space of its kernel. For instance, one can see \cite{bianchi1991note}. For similar approaches on fractional Sobolev inequality, one can see \cite{chen2013remainder}. There might be possible to prove Theorem \ref{intro:thm:local} using the non-degeneracy result we have shown. Since using Fourier series is more direct, we did not pursue this direction.

To get an example violating the uniform quantitative stability, we choose to perturb $\psi_{\vec{\alpha}}$ with $\alpha_1=j^2+\im$ and $\alpha_2=-j^2+\im$. For any $u$ near to $\psi_{\vec{\alpha}}$, we formally decompose it to the sum of one part in the kernel of linearized operator at $\psi_{\vec{\alpha}}$ and the other part $u_\perp$ in the orthogonal space. If $u_\perp$ is almost orthogonal to the kernel of the linearized operator at $\psi_{\alpha_1}$ or $\psi_{\alpha_2}$, then one still gets \eqref{stab-1/2}.
As $j\to \infty$, $\psi_{\alpha_1}$ and $\psi_{\alpha_2}$ has very weak interaction. The four elements of $\ker L_{\psi_{\vec{\alpha}}}$ split to two elements of $\ker L_{\psi_{\alpha_1}}$ and two elements of $\ker L_{\psi_{\alpha_2}}$. However, the element $1\in\ker L_{\psi_{\vec{\alpha}}} $, which corresponds to the rotation, can not split. This gives us some hope to find $u_\perp$ which is almost orthogonal to $\ker L_{\psi_{\vec{\alpha}}}$ but lies in $\ker L_{\psi_{\alpha_1}}$ and $\ker L_{\psi_{\alpha_2}}$ approximately. One typical example is  that $u$ is 1 near $j^2$  and is $-1$ near $-j^2$, which are centers of  $\psi_{\alpha_1}$ and $\psi_{\alpha_2}$ respectively, as constructed in \eqref{fj}. Once realizing this, the only job left is making sure the infimum in \eqref{intro:inf>M} can be achieved by $\psi_{\vec{\alpha}}$. This is done by using the implicit function theorem near $\psi_{\vec{\alpha}}$. We are inspired by the modulation analysis near Talenti bubbles, for instance see \cite{collot2017dynamics}.

\subsection{Structure of the paper} In Section \ref{sec:pre}, we give a detailed preliminary on $\DRh$, $\cnr$ and half-harmonic maps. The equivalence of half-harmonic maps defined on $\R$ and $\S$ is used implicitly in the following sections. In Section \ref{sec:non-d}, we prove the non-degeneracy of the linearized operator at each half-harmonic map. We divert to consider the stability from Section \ref{sec:stab-1}. There we prove the quantitative stability for degree $\pm1$ and a local result for higher degree. Section \ref{sec:higher} devotes to constructing an example losing uniform stability. Finally, we put some tedious computations in the Appendix which are needed in Section \ref{sec:higher}.

\section{Preliminary}\label{sec:pre}
In this section, we lay some foundations for half-harmonic maps. There are various perspective to define them.
\subsection{Formulations on the real line}
\begin{definition}
For any $f:\R\to \R$, we define
\begin{align}\label{def:1/2-R}
    \DRh f(x)=\frac{1}{\pi}P.V.\int_{\R}\frac{f(x)-f(y)}{|x-y|^2}dy.
\end{align}
\end{definition}
We call $f\in \dot H^{1/2}(\R)$ if
\begin{align}\label{R-1/2-norm}
    \|f\|_{\dot H^{1/2}(\R)}^2:=\frac{1}{4\pi^2}\iint_{\R\times \R}\frac{|f(x)-f(y)|^2}{|x-y|^2}dxdy<\infty.
\end{align}
Suppose $\u=(u_1,u_2)$ is a map from $\R$ into $\S$. Here and the following we always assume $\S$ is embedded in $\R^2=\mathbb{C}$.
Define the energy of $\u:\R\to \S$ by
\begin{align}\label{1.1}
   \mathcal{E}(\u):= \frac{1}{2\pi}\int_{\mathbb{R}}\u\cdot \DRh \u d x.
\end{align}
Using $|\u|=1$, it is easy to verify that
\begin{align}
    \u\cdot \DRh \u (x)=\frac{1}{2\pi}P.V.\int_{\R}\frac{|\u(x)-\u(y)|^2}{|x-y|^2}dy.
\end{align}
Consequently
\begin{align}
\mathcal{E}(\u)=\|\u\|_{\dot{H}^{1/2}(\R)}^2=\|u_1\|_{\dot{H}^{1/2}(\R)}^2+\|u_2\|_{\dot{H}^{1/2}(\R)}^2.
\end{align}
The functional $\mathcal{E}$ is invariant under the trace of conformal maps keeping invariant the half-space $\mathbb{R}^2_+$: the M\"obius group.

The critical points of $\mathcal{E}$ are called \textit{half-harmonic maps}.
\begin{definition} A map $\u \in \dot{H}^{1 / 2}(\mathbb{R}, \S)$ is called a weak half-harmonic map if for any $\boldsymbol{\phi} \in \dot{H}^{1 / 2}\left(\mathbb{R}, \R^2\right) \cap$ $L^{\infty}\left(\mathbb{R}, \R^2\right)$ there holds
\[
\frac{d}{d t}\bigg|_{t=0} \mathcal{E}\left(\frac{\u+t \boldsymbol{\phi}}{|\u+t\boldsymbol{\phi}|}\right)=0.
\]
\end{definition}
Computing the associated Euler-Lagrange equation of \eqref{1.1}, we obtain that if $\u: \mathbb{R} \rightarrow \S$ is a half-harmonic map, then $\u$ satisfies the following equation:
\begin{align}\label{E-L-R}
\left(-\Delta_{\mathbb{R}}\right)^{\frac{1}{2}} \u(x)=\left(\u\cdot \DRh \u\right) \u(x) \quad\text { in } \mathbb{R} .
\end{align}

All the half-harmonic maps (with finite energy) have been classified by \citet{millot2015fractional}.
In the sequel, we identify $\mathbb{R}^{2}$ with the complex plane $\mathbb{C}$ writing $z=x_{1}+\im x_{2}$. We shall write $u=u_1+\im u_2$ for any map $\u=(u_1,u_2)$.

\begin{theorem}[\citet{millot2015fractional}]\label{thm:class}
Let $u \in \dot{H}^{1 / 2}\left(\mathbb{R} ; \S\right)$ be a non-constant half-harmonic map into $\S$. Let $U$ be the harmonic extension of $u$ to $\mathbb{R}_{+}^{2}$. There exist some $d \in \mathbb{N}, \vartheta \in \mathbb{R}$ and $\left\{\alpha_{k}\right\}_{k=1}^{d} \subseteq \mathbb{H}=\{z\in \mathbb{C}:\textup{Im\,}z>0\}$ such that $U(z)$ or its complex
conjugate equals
\begin{align}\label{1/2-H-form}
\psi_{\vartheta,\vec{\alpha}}:=e^{i \vartheta} \prod_{k=1}^{d} \frac{z-\alpha_k}{z-\bar\alpha_k} .
\end{align}
Furthermore,
\begin{align}\label{eng-u-d}
    \|u\|_{\dot H^{1 / 2}(\mathbb{R})}^{2}=\frac{1}{2\pi} \int_{\mathbb{R}_{+}^{2}}\left|\nabla U\right|^{2} d z=d.
\end{align}
\end{theorem}
The above theorem says that the Stereographic projection
\begin{align}\label{def:S(x)}
    \mathcal{S}(x)=\left(\frac{2 x}{x^{2}+1}, \frac{x^{2}-1}{x^{2}+1}\right): \mathbb{R} \rightarrow \S
\end{align}
is a half-harmonic map. Actually one can verify it directly. It follows from \eqref{def:1/2-R} and some computations that
\begin{align}
(-\Delta)^{\frac{1}{2}} \mathcal{S}(x)=\frac{1}{\pi} P . V . \int_{\mathbb{R}} \frac{\mathcal{S}(x)-\mathcal{S}(y)}{|x-y|^{2}} d y&=\left(\frac{4 x}{\left(1+x^{2}\right)^{2}}, \frac{2\left(x^{2}-1\right)}{\left(1+x^{2}\right)^{2}}\right), \\
\frac{1}{2 \pi} \int_{\mathbb{R}} \frac{|\mathcal{S}(x)-\mathcal{S}(y)|^{2}}{|x-y|^{2}} d y&=\frac{2}{1+x^{2}}.
\end{align}
It is easy to verify that $\mathcal{S}$ satisfies \eqref{E-L-R}. Therefore it is a half-harmonic map from $\R\to \S$.
\subsection{Formulations on the unit circle}
Using the Stereographic projection in \eqref{def:S(x)},
one can reformulate the problem by the maps from $\S$ to $\S$.
To that end, we parametrize $\S=\{e^{\im\theta}:\theta\in[0,2\pi)\}$. Let $\mathcal{S}(x)=e^{\im\theta}$, $\mathcal{S}(y)=e^{\im\vartheta}$, then
\[|e^{\im\theta}-e^{\im\vartheta}|^2=|\mathcal{S}(x)-\mathcal{S}(y)|^2=\frac{4(x-y)^2}{(x^2+1)(y^2+1)}.\]
Using $x=\mathcal{S}^{-1}(\theta):=\mathcal{S}^{-1}(e^{\im \theta})$, we can write $f(x)$ defined on $\R$ to $\tilde f(\theta)$ defined on $\S$. By the above identity and \eqref{def:1/2-R}, we obtain
\begin{align}\label{1/2-relation-R-S}
    (-\Delta_{\R})^{\frac12}f(x)=\frac{1}{\pi}P.V.\int_{\R}\frac{f(x)-f(y)}{|x-y|^2}dy=\frac{2}{1+x^2}\frac{1}{\pi}P.V.\int_{\S}\frac{\tilde f(\theta)-\tilde f(\vartheta)}{|e^{\im\theta}-e^{\im\vartheta}|^2}d\vartheta.
\end{align}
This leads to the definition of $\DSh$.
\begin{definition}For any $\tilde f:\S\to \R$, we define the
\begin{align}\label{def:1/2-S}
(-\Delta_{\S})^{\frac12}\tilde f(\theta):=\frac{1+x^2}{2}\DRh f(x)=\frac{1}{\pi}P.V.\int_{\S}\frac{\tilde f(\theta)-\tilde f(\vartheta)}{|e^{\im\theta}-e^{\im\vartheta}|^2} d\vartheta.
\end{align}
\end{definition}
We call $\tilde f\in \cn$ if and only if
\begin{align*}
    \|\tilde f\|^2_{\dot H^{1/2}(\S)}:=
    \frac{1}{4\pi^2}\iint_{\S\times \S}\frac{|\tilde f(\theta)-\tilde f(\vartheta)|^2}{|e^{\im\theta}-e^{\im\vartheta}|^2}d\vartheta d\theta<\infty.
\end{align*}
Using \eqref{1/2-relation-R-S}, it is easy to see that for any $f:\R\to \S$,
\begin{align}\label{isometry}
    \| f\|_{\dot H^{1/2}(\R)}=\| f\circ \mathcal{S}^{-1}\|_{\cn}.
\end{align}
Therefore $\mathcal{S}$ is an isometric isomorphism of $\dot H^{1/2}(\R)$ and $\cn$.

Now for any map $\u=(u_1,u_2):\S\to \S$, we still adopt the notation $u=u_1+\im u_2$ is a complex-valued function, we define $\DSh u=\DSh u_1+\im \DSh u_2$ and
\begin{align}\label{uS1n}
\|u\|_{\dot H^{1/2}(\S)}^2:=\|u_1\|_{\dot H^{1/2}(\S)}^2+\|u_2\|_{\dot H^{1/2}(\S)}^2=\Re\int_{\S} \bar u\DSh u.
\end{align}
Here we have used the multiplication of two complex numbers and Re denotes the real part. One readily derive that  $\|\text{id}_{\S}\|_{\dot H^{1/2}(\S)}^2=1$.

One also define the energy as $\mathcal{E}(u)=\|u\|_{\cn}^2$ and the the critical points of $\mathcal{E}$ are called half-harmonic maps $\S\to \S$, which satisfy
\begin{align}\label{E-L-S}
    \DSh u(\theta)=\left(\frac{1}{2\pi}P.V.\int_{\S}\frac{|u(\theta)-u(\vartheta)|^2}{|e^{\im\theta}-e^{\im\vartheta}|^2} d\vartheta\right)u(\theta)\quad \text{on }\S.
\end{align}
The isometry in \eqref{isometry} infers the one-to-one correspondence of half-harmonic maps between $\R\to \S$ and $\S\to \S$ through the stereographic projection.

One of the great advantages of working on $\S$ is that we have the Fourier expansion. Namely for any $u:\S\to \mathbb{C}$, we formally have
\[u=\sum_{-\infty}^\infty c_ke^{\im k\theta}=\sum_{k\in \mathbb{Z}}c_kz^k,\quad \text{where }c_k=\fint_{\S}uz^{-k}d\theta.\]
Using the fact that $\DSh z^k=|k|z^k$ for any $k\in \mathbb{Z}$ (for instance, see \cite{sire2017nondegeneracy}), we have the following interpretation of $\DSh$,
\[(-\Delta_{\S})^{\frac12}u=\sum_{k\in \Z} |k|c_k z^k,\]
consequently \eqref{uS1n} implies
\begin{align}\label{2norm-series}
    \|u\|_{\cn}^2=\sum_{k\in\Z}|k||c_k|^2.
\end{align}
Furthermore, if $u$ maps $\S$ to $\S$, then $|u|=1$ is equivalent to
\begin{align}\label{u=1-series}
    \sum_{j\in \mathbb{Z}} |c_j|^2=1,\quad \sum_{j\in \mathbb{Z}}\bar c_jc_{j+k}=0, \quad \forall \,k\in \mathbb{Z}.
\end{align}

Recall that the degree (or winding number) is well defined for $\dot H^{1/2}(\S;\S)$ by \citet{de1991boundary} (see also \citet{brezis1995degree}). More precisely,
\begin{align}\label{deg-series}
    \deg u=\frac{1}{2\pi \im}\int_{\S}\bar{u}du=\frac{1}{2\pi}\int_{\S}u\wedge u_\theta d\theta=\sum_{-\infty}^\infty k|c_k|^2.
\end{align}
Throughout this paper, we use the notation $(z_1+\im z_2)\wedge (w_1+\im w_2)=z_1w_2-z_2w_1$.
From the above equation and \eqref{2norm-series}, one can see that the degree is continuous in $\dot H^{1/2}(\S;\S)$ topology. For $u\in\dot H^{1/2}(\R;\S)$, we shall define $\deg u=\deg u\circ\mathcal{S}^{-1}$. We have the following simple fact.
\begin{lemma}\label{lem:deg_prod}
If $u,v\in \cnS$, then $\deg (uv)=\deg u+\deg v$.
\end{lemma}
\begin{proof} Since $u,v$ take values in $\S$, then using \eqref{deg-series} to obtain
\begin{align*}
    \deg (uv)=\frac{1}{2\pi \im}\int_{\S}\bar u\bar vd(uv)=\frac{1}{2\pi \im}\int_{\S}\bar udu +\bar vdv=\deg u+\deg v.
\end{align*}

\end{proof}

Define $\D=\{\xi=x+\im
y:|\xi|<1\}$ and consider $\S=\partial \D$. Recall that $f\in \dot H^{1/2}(\S)$ if and only if it is the trace of some function in $\dot H^1(\D)$. The energy for $u\in \dot H^{1/2}(\S;\S)$ has a tractable representation
\begin{align}\label{norm-ext-disk}
    \|u\|_{\cn}^2=\inf\left\{\frac{1}{2\pi}\int_{\D}|\nabla U|^2:U\in \dot H^1(\D;\D) \text{ with } \text{Tr\,} U=u\right\}.
\end{align}
This infimum is achieved by the harmonic extension of $u$. Since $\partial U/\partial\nu=\DSh u$ on $\S$, where $\nu$ is the outward unit normal of $\partial\D$, then \eqref{E-L-S} is equivalent to
\begin{align}
\begin{cases}
    \Delta U=0 &\text{in }\D,\\
    \frac{\partial U}{\partial \nu}\wedge U=0 &\text{on }\S.
\end{cases}
\end{align}

All half-harmonic maps, as classified in Theorem \ref{thm:class}, are Blaschke products of $d$ M\"obius transformations of $\D$ or their complex conjugate. This can seen from  Cayley transformation. Recall that \textit{Cayley transform} is a bi-holomorphic mapping $\mathfrak{C}: {\mathbb{H}} \rightarrow {\mathbb{D}}$ which is defined as
\begin{align}\label{cayley}
    \xi=\mathfrak{C}(z):=\im \frac{z-\im}{z+\im}.
\end{align}
Here and the following, we use $\xi$ to denote the complex coordinates of $\D$ and $z$ to denote that of $\mathbb{H}$. Note that
the boundary mapping of $\mathfrak{C}(z)$ is just the stereographic projection \eqref{def:S(x)}.

By the one-to-one correspondence of half-harmonic maps in $\dot H^{1/2}(\R;\S)$ and $\dot H^{1/2}(\S;\S)$, one can  obtain all half-harmonic maps in $\dot H^{1/2}(\S;\S)$ from \eqref{1/2-H-form}.
More precisely, it consists of Blaschke products $\{\phi_{\vartheta',\vec{a}}:\vartheta'\in \S, \,\vec{a}=(a_1,\cdots,a_d)\in \D^d\}$ and their complex conjugates, where
\begin{align}\label{def:phi}
    \phi_{\vartheta',\vec{a}}(\xi):=e^{i\vartheta'}\prod_{j=1}^d\frac{\xi-a_j}{1-\bar a_j \xi}.
\end{align}
The parameters in \eqref{1/2-H-form} and the above equation are related by
\begin{align}
    a_j=\mathfrak{C}(\alpha_j),\quad  \vartheta'=\vartheta+\sum_{j=1}^d\theta_j,\quad e^{\im \theta_j}=\frac{-\alpha_j-\im}{1+\im\bar \alpha_j}.
\end{align}

Using some integration by parts, the degree of $u$ (see \eqref{deg-series}) can also be defined by
\begin{align}\label{deg-disk}
    \deg u=\frac{1}{\pi}\int_{\D}U_x\wedge U_y dxdy.
\end{align}
Notice the following identity when $U=U^1+\im U^2$
\begin{align}\label{id-U}
    |\nabla U|^{2}=\left(U_{x}^{1}-U_{y}^{2}\right)^{2}+\left(U_{y}^{1}+U_{x}^{2}\right)^{2}+2 U_{x} \wedge U_{y} \geqslant 2 U_{x} \wedge U_{y}.
\end{align}
Since $\phi_{\vartheta,\vec{a}}$ is holomorphic in $\D$, $\phi_{\vartheta,\vec{a}}$ achieves the identity for the above equation. Combining this with  \eqref{norm-ext-disk} and \eqref{deg-disk}, we just proved \eqref{eng-u-d} since
\begin{align}
    \|\phi_{\vartheta,\vec{a}}\|_{\cn}^2=\deg \phi_{\vartheta,\vec{a}}=d.
\end{align}

We show the proof of Theorem \ref{intro:lower-bd} to end this section.
\begin{proof}[{\bf Proof of Theorem \ref{intro:lower-bd}}]
The inequality is trivial by \eqref{id-U}. The proof of equality is essentially contained in \cite{mironescu2004variational}. For the reader's convenience, we include it here. Suppose $U$ is the harmonic extension of $u$. It follows from $\|u\|_{\cn}=|\deg u|$ and \eqref{id-U} that
\[U_x^1=U_y^2\quad \text{ and }\quad U_y^1=U^2_x\quad \text{ in }L^2(\D).\]
By Weyl’s lemma, this is equivalent to $U$ being holomorphic from $\D$ to $\D$. Moreover, $|U(z)|\to 1$ uniformly as $|z|\to 1$ (see \cite{brezis1995degree}). Then the result of  \citet{burckel1980introduction} implies all such maps  must be Blaschke products.
\end{proof}




\section{Non-degeneracy of the linearized operator}\label{sec:non-d}

Consider the half-harmonic map $u=u_1+\im u_2:\R\to \S$, equation \eqref{E-L-R} is equivalent to
\begin{align}
    \DRh u\wedge u=0.
\end{align}
The linearized operator at $u$ is
\begin{align}
    L_u[v]=\DRh u\wedge v+\DRh v\wedge u.
\end{align}
We define
\[\ker L_u=\{v=v_1+\im v_2\in \dot H^{1/2}(\R;\mathbb{C}):L_u[v]=0, u_1v_1+u_2v_2=0\}.\]
Suppose $V=V_1+\im V_2$ is the harmonic extension of $v$ to the upper half space $\mathbb{H}$ by Possion kernel, then $v\in \ker L_u$ is equivalent to
\begin{align}\label{eq:V-Rp}
    \begin{cases}
    \Delta V=0&\text{in }\mathbb{H},\\
   \partial_y U\wedge v+\partial_y V\wedge u=0&\text{on }\partial \mathbb{H}.
    \end{cases}
\end{align}



The stereographic projection induces an isometry between $\dot H^{1/2}(\R)$ and $\cn$. For a half-harmonic map $u:\S\to \S$, one can also define the linearized operator
\begin{align}
    \tilde L_u[v]=\DSh u\wedge v+\DSh v\wedge u
\end{align}
and $\ker \tilde L_u=\{v\in\dot H^{1/2}(\S;\mathbb{C}):\tilde L_u[v]=0,\Re u\bar v=0\}$.

\begin{lemma}\label{lem:kerR-kerS}
For any half-harmonic map $u:\R\to \S$, let $\tilde u=u\circ\mathcal{S}^{-1}$, we have
\[L_{ u}=J_{\mathcal{S}}(\tilde L_{\tilde u}\circ \mathcal{S}^{-1})\]
where $J_\mathcal{S}=(1-\sin\theta)$ is the Jacobian of $\mathcal{S}$. Consequently $\text{dim}_{\R}\ker L_{ u}=\text{dim}_{\R}\ker \tilde L_{\tilde u}$.
\end{lemma}
\begin{proof}It is easy to show that the Jacobian of $\mathcal{S}$ is
\[J_{\mathcal{S}}=\frac{2}{x^2+1}=1-\sin\theta.\]
Recall from \eqref{def:1/2-R} that
\[\DRh u(\mathcal{S}(\theta))=(1-\sin\theta){\DSh  \tilde u(\theta)}.\]
For any $v\in \dot H^{1/2}(\R;\mathbb{C})$, denote $\tilde v=v\circ \mathcal{S}^{-1}\in \dot H^{1/2}(\S;\mathbb{C})$, then
\[L_{ u}[ v](x)=\DRh u\wedge  v+\DRh  v\wedge  u={(1-\sin\theta)}\tilde L_{\tilde u}[\tilde v].\]
\end{proof}

The operator $L_u$ (or $\tilde L_u$) arise naturally in the linearization of the energy functional $\mathcal{E}$.
To see that, let us use $\u^\perp$ denote the vector rotating $\u$ counterclockwise by $\pi/2$. Any variation $\boldsymbol{\phi}$ can be written as $\tilde h\u+h\u^\perp$ for some real-valued function $\tilde h$ and $h$ in $\dot H^{1/2}(\R)$.
\begin{lemma}\label{lem:vari-E}
Suppose that $\u$ is a half-harmonic map $\R\to \S$ and $\boldsymbol{\phi}=\tilde h\u+h\u^\perp$ is a variation satisfying $|\u+\boldsymbol{\phi}|=1$. Then
\begin{align}\label{energ-expand}
    \mathcal{E}(\u+\boldsymbol{\phi})=\mathcal{E}(\u)+\int_{\R} h\mathcal{L}_{\u}[h]+O(h^3)
\end{align}
where the operator $\mathcal{L}_{\u}$ is given by
\begin{align}\label{def:calL}
    \mathcal{L}_{\u}=\DRh-(\u\cdot \DRh \u)+R
\end{align}
with the integral operator
\begin{align}\label{def:R}
    (Rf)(x)=\frac{1}{2\pi}\int_{\R}\frac{|\u(x)-\u(y)|^2}{|x-y|^2}f(y)dy.
\end{align}
\end{lemma}
\begin{proof}
The constraint $|\u+\boldsymbol{\phi}|^2=1$ a.e. implies that $2\tilde h+\tilde h^2+h^2=0$. Then $\tilde h=-\frac{1}{2}h^2+O(h^3)$. We have
\begin{align}\begin{split}\label{energ-expand-1}
    \mathcal{E}(\u+\boldsymbol{\phi})=&\ \mathcal{E}(\u)+2\int_{\R} \boldsymbol{\phi}\cdot \DRh \u+\int_{\R}\boldsymbol{\phi}\cdot\DRh \boldsymbol{\phi}\\
    =&\ \mathcal{E}(\u)-\int_{\R}(\u\cdot \DRh \u) h^2+\int_{\R}\boldsymbol{\phi}\cdot\DRh \boldsymbol{\phi}+O(h^3).
\end{split}\end{align}
It follows from \eqref{def:1/2-R} that
\begin{align}
    \DRh (h\u^\perp)(x)=(\DRh h(x))\u^\perp(x)+\frac{1}{\pi}P.V.\int_{\R}\frac{\u^\perp(x)-\u^\perp(y)}{|x-y|^2}h(y)dy.
\end{align}
Using $\u^\perp (x)\cdot(\u^\perp(x)-\u^\perp(y))=1-\u(x)\cdot \u(y)=\frac{1}{2}|\u(x)-\u(y)|^2$, we conclude that
\begin{align*}
    \int_{\R}\boldsymbol{\phi}\cdot\DRh\boldsymbol{\phi}=&\int_{\R}(\tilde h\u+h\u^\perp)\cdot \DRh(\tilde h\u +h\u^\perp)\\
    =&\int_{\R}h\u^\perp\cdot \DRh (h\u^\perp)+O(h^3)\\
    =&\int_{\R}h\DRh h+\frac{1}{2\pi}\iint_{\R\times\R}\frac{|\u(x)-\u(y)|^2}{|x-y|^2}h(x)h(y)dxdy+O(h^3).
\end{align*}
Plugging it into \eqref{energ-expand-1}, we get \eqref{energ-expand}.
\end{proof}

For any $v=v_1+\im v_2\in \dot H^{1/2}(\R;\mathbb{C})$ satisfying $u_1v_1+u_2v_2=0$, we can write $v=h\im u $ for some real-valued function $h$ on $\R$. Then one can verify that
\begin{align}\nonumber
    \begin{split}
         &L_u[v]=\DRh u\wedge (h\im u)+\DRh (h\im u)\wedge u\\
     =& [\DRh u\wedge \im u]h-\DRh h(\im u\wedge u)+\frac{1}{\pi}P.V.\int_{\R}\frac{(\im u(x)-\im u(y))\wedge u(x)}{|x-y|^2}h(y)dy\\
     =& (\u\cdot\DRh \u)h-\DRh h-Rh\\
     =&-\mathcal{L}_{\u}[h].
    \end{split}
\end{align}

\begin{lemma}\label{lem:kerR-dayu}
Suppose $u$ is a half-harmonic map from $\R$ to $\S$, then $\text{dim}_{\R} \ker  L_u\geq2|\deg u|+1$. More precisely, for instance, if $u=e^{\im \vartheta}\prod_{j=1}^k(\frac{x-\alpha_j}{x-\bar \alpha_j})^{d_j}$ with $\{\alpha_j\}_{j=1}^k$ are distinct and $d_j\geq 1$, then
\begin{align*}
    \ker L_{u} \supset\textup{span}_{\R}\left\{1,\Re \frac{1}{(x-\bar \alpha_j)^{s}},\Im \frac{1}{(x-\bar\alpha_j)^{s}}: s=1,\cdots,d_j,j=1,\cdots,k.\right\}\im u.
\end{align*}
\end{lemma}
\begin{proof} Assume $d=\deg u>0$.
It is known that
$u$ takes form \eqref{intro:1/2-H-form} if and only if there exists $c_0,\cdots,c_{d-1}\in \mathbb{C}$ and $\vartheta\in \S$ such that
\begin{align}\label{rational-R}
    u=e^{\im \vartheta}\frac{x^d+c_{d-1}x^{d-1}+\cdots+c_1x+c_0}{x^d+\bar c_{d-1}x^{d-1}+\cdots+\bar c_1x+\bar{c}_0}
\end{align}
with $x^d+c_{d-1}x^{d-1}+\cdots+c_1x+c_0$ has zeros all in $\mathbb{H}$. This fact comes from the theorem 3.3.2 in \cite{garcia2018finite} (page 43) and conformal equivalence between $\D$ and $\mathbb{H}$.

It is clear that changing the parameters $\vartheta,c_{d-1},\cdots,c_0$ (complex numbers) of $u$ continuously
yields a family of half-harmonic maps. Therefore, it generates kernel maps for the linearized operators $ L_u$. Taking derivatives of $u$ on $\vartheta$, we get $\{\im u\}\subset \ker L_u$. It is easy to see that for $l=0,\cdots,d-1$,
\begin{align}
   & \partial_{c_j}u=e^{\im \vartheta}\frac{x^{l}}{x^d+\bar c_{d-1}x^{d-1}+\cdots+\bar c_1x+\bar c_0}=\frac{x^{l}u}{\prod_{j=1}^k (x-\alpha_j)^{d_j}},\\
    &\partial_{\bar c_j}u=\frac{-x^{l}u}{x^d+\bar c_{d-1}x^{d-1}+\cdots+\bar c_1x+\bar c_0}=\frac{-x^{l}u}{\prod_{j=1}^k (x-\bar\alpha_j)^{d_j}}.
\end{align}
Note that $x\in\R$. Taking derivative of $u$ on the real part of $c_j$, we get
\begin{align}
\begin{split}\label{ker-real-pt}
    \partial_{c_j}u+\partial_{\bar c_j}u
    = 2\im u\  \text{Im\,}\frac{x^{l}}{\prod_{j=1}^k (x-\bar\alpha_j)^{d_j}}.
\end{split}
\end{align}
Taking derivative of $u$ on the imaginary part of $c_j$, we get
\begin{align}\label{ker-img-pt}
    \im(\partial_{c_j}u-\partial_{\bar c_j}u)= 2\im u\ \Re\frac{x^{l}}{\prod_{j=1}^k (x-\bar\alpha_j)^{d_j}}.
\end{align}
Therefore, the $\R$-linear combination of them must belong to $\ker  L_u$.
\begin{align*}
    \begin{split}
        \ker  L_u\supset&\ \text{span}_{\R}\left\{1,\Re\frac{x^{l}}{\prod_{j=1}^k (x-\bar\alpha_j)^{d_j}},\text{Im\,}\frac{x^{l}}{\prod_{j=1}^k (x-\bar\alpha_j)^{d_j}}:l=0,\cdots,d-1\right\}\im u\\
        =&\ \textup{span}_{\R}\left\{1,\Re \frac{1}{(x-\bar \alpha_j)^{s}},\Im \frac{1}{(x-\bar\alpha_j)^{s}}: s=1,\cdots,d_j,j=1,\cdots,k.\right\}\im u.
    \end{split}
\end{align*}
Obviously they are all $\R$-linearly independent, therefore $\text{dim}_{\R}\ker  L_u\geq 2|d|+1$.

If $d=0$, then obviously $\{\im u\}\subset\ker  L_u$ and thus $\dim_{\R}\ker L_u\geq 1$. If $d<0$, one can prove similarly as above by working on the conjugate of $u$.
\end{proof}

Suppose that $\boldsymbol{V}=(V_1,V_2)$. We abuse the notation $V=V_1+\im V_2$ and denote it a complex-valued function defined on upper half plane $\mathbb{H}$. We also adopt the notation $\partial_z=\frac{1}{2}(\partial_x-\im\partial_y)$, $\partial_{\bar z}=\frac{1}{2}(\partial_x+\im\partial_y)$.

\begin{lemma}\label{lem:V-holo}
Suppose $u$ is half-harmonic map from $\R$ to $\S$.
If $v\in \ker L_u$  with $\deg u>0$ ($\deg u<0$), then $V$ is holomorphic (anti-holomorphic) in $\mathbb{H}$. Moreover, $V$ can be extended to a meromorphic (anti-meromorphic) function with possible poles at poles of $U$. In addition, $V\circ \mathfrak{C}^{-1}$ is smooth on $\overline{\D}$.
\end{lemma}

\begin{proof}
We just prove the case $\deg u>0$. Define the Hopf differential
\begin{align}
    \Phi=2U_z\bar V_z=\partial_x \boldsymbol{U}\cdot \partial_x \boldsymbol{V}-\partial_y \boldsymbol{U}\cdot \partial_y \boldsymbol{V}-\im(\partial_x \boldsymbol{U}\cdot \partial_y \boldsymbol{V}+\partial_x \boldsymbol{V}\cdot \partial_y \boldsymbol{U}).
\end{align}
Here in the middle we are using complex multiplication and $\bar V_z=\partial_z \bar V$. Since $U$ is holomorphic, then
\[\Phi_{\bar z}=2U_{z\bar z}\bar V_z+2U_z\bar V_{z\bar z}=2U_z\overline{V_{z\bar z}}=\frac{1}{2}U_z\overline{\Delta V}=0.\]
Therefore $\Phi$ is holomorphic in $\mathbb{H}$. We claim that $\text{Im}\, \Phi=0$ on $\partial \mathbb{H}$. To see that, the boundary condition  in \eqref{eq:V-Rp} means
\begin{align}
    0=\partial_yV\wedge u+(\u\cdot \partial_y\boldsymbol{U})u\wedge v=(\partial_y V-(\u\cdot \partial_y \boldsymbol{U})v)\wedge u\quad \text{on }\partial \mathbb{H},
\end{align}
which means
\begin{align}\label{V-bd}
    \partial_yV=(\u\cdot \partial_y \boldsymbol{U})v+(\u\cdot \partial_y \boldsymbol{V})u,\quad\text{ on }\partial \mathbb{H}.
\end{align}
Since $\partial_x\u\cdot \u=0$ and $\partial_x \u\cdot \boldsymbol{v}+\u\cdot \partial_x\boldsymbol{v}=0$ on $\partial \mathbb{H}$, then
\begin{align}
    \partial_x \u\cdot \partial_y \boldsymbol{V}+\partial_x \boldsymbol{v}\cdot \partial_y \boldsymbol{U}=(\u\cdot \partial_y\boldsymbol{U})\partial_x\u\cdot \boldsymbol{v}+(\u\cdot \partial_y \boldsymbol{U})\partial_x \boldsymbol{v}\cdot \u=0.
\end{align}

Thus we have shown that $g(z):=\text{Im}\,\Phi$ vanishes on $\partial \mathbb{H}$. By odd reflection across $\partial \mathbb{H}$, we can extend the harmonic function $g$ to all of $\mathbb{C}$. However, since $g$ is harmonic and $g\in L^1(\Rp)$ because $U,V\in \dot H^1(\Rp)$, we conclude that $g\equiv 0$ on $\mathbb{C}$. Thus $\Phi$ is real-valued and holomorphic, which implies that $\Phi$ is constant. Since $\Phi\in L^1(\Rp)$, we deduce that $\Phi(z)\equiv 0$.

Now we have $\Phi=2U_z\bar V_z=0$. Since $U_z$ only have isolated zeros in $\mathbb{H}$, then $\bar V_z=0$. Consequently, $V$ is holomorphic.

We shall rewrite $U$ as
\begin{align}\label{U-different}
    U(z)=e^{\im\vartheta}\prod_{j=1}^k\left(\frac{z-\alpha_j}{z-\bar \alpha_j}\right)^{d_j}
\end{align}
with $\alpha_j\in \mathbb{H}$ which are different from each other. Then $d_1+\cdots+d_k=d$.
Now define $W=V/U$. Then the previous argument implies that $W$ is a meromorphic function on $\mathbb{H}$ with possible poles at $\{\alpha_1,\cdots,\alpha_k\}$. The orthogonality condition says that $u_1v_1+u_2v_2=0$ on $\partial \mathbb{H}$. Therefore the real part of $ W$ vanishes on $\partial\mathbb{H}$ because
\begin{align}
    W=\frac{v_1+\im v_2}{u_1+\im u_2}=\frac{u_1v_1+u_2v_2}{u_1^2+u_2^2}+\im\frac{u_1v_2-u_2v_1}{u_1^2+u_2^2}.
\end{align}

By Schwarz reflection principle, we can extend $W$ to a meromorphic function on $\mathbb{C}$ (we still denote it as $W$) by $
W(z)=-\overline{W(\bar z)}$ for the lower half plane. Now $W$ has possible poles at $\{\alpha_1,\cdots,\alpha_k,\bar{\alpha}_1,\cdots,\bar{\alpha}_k\}$. The order of $W$ at $\alpha_j$ (or $\bar \alpha_j$) is at most $d_j$. Since $U$ is a meromorphic function on $\mathbb{C}$, then so does $V=WU$. The poles of $V$ are contained in $\{\bar{\alpha}_1,\cdots,\bar{\alpha}_k\}$.

Note that $\bar\alpha_j$, $j=1,\cdots,k$, are all away from the $\partial \mathbb{H}$. Therefore $V$ is holomorphic at any point on $\partial \mathbb{H}$. Since $\mathfrak{C}$ is a holomorphic on $\overline{\mathbb{H}}$, then $V\circ \mathfrak{C}^{-1}$ is holomorphic in $\overline{\mathbb{D}}\setminus\{\text{i}\}$. Because $\dot H^{1/2}(\R;\mathbb{C})$ is isometric to $\dot H^{1/2}(\S;\mathbb{C})$ through the stereographic projection, then $v\circ \mathfrak{C}^{-1}\in \dot H^{1/2}(\S;\mathbb{C})$. Thus $v\circ \mathfrak{C}^{-1} \in \ker \tilde L_{u\circ \mathfrak{C}^{-1}}$.
However, we can repeat the above whole process by using another bi-holomorphic mapping between $\mathbb{H}$ and $\mathbb{D}$, say $(z-\im)/(z+\im)$, to show that $V\circ \mathfrak{C}^{-1}$ is holomorphic in $\overline{\mathbb{D}}\setminus\{1\}$. Combining with the previous statement, we know $V\circ \mathfrak{C}^{-1}$ is holomorphic in $\overline{\mathbb{D}}$.
\end{proof}


Now we can prove the Theorem \ref{intro:thm:non}.
\begin{proof}[{\bf Proof of Theorem \ref{intro:thm:non}}]
We shall assume $\deg u>0$ and $U$ takes the form \eqref{U-different}. Suppose $v\in\ker L_u$ and $V$ is the harmonic extension to $\Rp$. Define $W=V/U$. Lemma \ref{lem:V-holo} implies that $W$ is meromorphic on $\mathbb{C}$ with possible poles at $\{\alpha_1,\cdots,\alpha_k,\bar{\alpha}_1,\cdots,\bar{\alpha}_k\}$. The order of $W$ at $\alpha_j$ (or $\bar \alpha_j$) is at most $d_j$.

We claim that $W$ is bounded at infinity. Indeed, this just follows from the fact that $V\circ \mathfrak{C}^{-1}$ is holomorphic in $\overline{\mathbb{D}}$ and consequently $V$ is bounded on $\overline{\mathbb{H}}$.

Thus $W$ is a rational function. There exists some polynomial $P(x)$ such that
\begin{align}
    W(z)=\frac{P(z)}{Q(z)},\quad Q(x)=\prod_{j=1}^k(z-a_j)^{d_j}(z-\bar a_j)^{d_j}.
\end{align}
The boundedness of $W$ implies $\deg P\leq 2d$. Since $W(z)=-\overline{W(\bar z)}$ and $Q(z)=\overline{Q(\bar{z})}$, it holds that $P(z)=-\overline{P(\bar z)}$. If $P(z)=c_0+c_1z+\cdots+c_{2d}z^{2d}$, then one must have $c_j=-\bar c_j$. The dimension of all such polynomials is $2d+1$. Thus $\dim_{\R}\ker L_u\leq 2|\deg u|+1$.  Applying  Lemma \ref{lem:kerR-kerS} and Lemma \ref{lem:kerR-dayu}, we conclude $\dim_{\R}\ker L_u=2|\deg u|+1$. Furthermore, if $u=e^{\im \vartheta}\prod_{j=1}^k(\frac{x-\alpha_j}{x-\bar \alpha_j})^{d_j}$ with $\{\alpha_j\}_{j=1}^k$ are distinct and $d_j\geq 1$, then
\begin{align*}
    \ker L_{u} =\textup{span}_{\R}\left\{1,\Re \frac{1}{(x-\bar \alpha_j)^{s}},\Im \frac{1}{(x-\bar\alpha_j)^{s}}: s=1,\cdots,d_j,j=1,\cdots,k.\right\}\im u.
\end{align*}
\end{proof}
\begin{remark}
Translating the above results to half-harmonic map $\tilde u:\S\to \S$ (see \eqref{def:phi}), we can get
\begin{align*}
\ker \tilde L_{\tilde u}=\textup{span}_{\R}\left\{1
, \textup{ Im\,}\frac{\xi^j}{\prod_{i=1}^d(\xi-a_i)},\textup{Re\,}\frac{\xi^j}{\prod_{i=1}^d(\xi-a_i)}:j=0,\cdots,d-1\right\}\im \tilde u.
\end{align*}
In particular, if $\tilde u=\xi^m$, then
\begin{align}
    \ker \tilde L_{\xi^m}=\im \xi^m\ \textup{span}_{\R}\{1, \textup{Im\,} \xi^{j-m},  \Re \xi^{j-m}:j=0,\cdots,m-1\}.
\end{align}
Using $\xi=\cos \theta+\im\sin\theta$ on $\S$, we can rewrite it as
\begin{align}
    \ker \tilde L_{\xi^m}=\im \xi^m\ \textup{span}_{\R}\{1,\sin j\theta,\cos j\theta:j=1,\cdots,m\}.
\end{align}
This recovers the result of \citet[Proposition 6.2]{enno2018energy}.
\end{remark}

\section{Stability of half-harmonic map}\label{sec:stab-1}
We only need to show the Theorem \ref{intro:thm:sta-type1} for $d=1$. The negative case follows from taking complex conjugate in the proof.  The case $d=\pm 1$ has a flavor of the main theorem proved by \citet{hirsch2021note} about the stability of harmonic maps $\mathbb{S}^2\to \mathbb{S}^2$ with degree $\pm1$. We will use induction to prove a local quantitative stability for higher degree.


Denote the set of all maps in $\dot H^{1/2}(\S;\S)$ with degree $d$ by
\begin{align}
    \mathcal{A}_{\S}^d=\{u\in\dot H^{1/2}(\S;\S):\deg\, u=d\}.
\end{align}
\begin{lemma}\label{lem:mean=0}
For any $u\in \mathcal{A}_{\S}^{d}$ with $|d|\geq 1$, there exists $a_0\in \D$ such that $\fint_{\S}u\circ\phi_{0,a_0} =0$.
\end{lemma}
\begin{proof}
Assume first $u\in \cnS\cap C^\infty$. Since $|\deg u|>1$, then $u$ is surjective. We will write $\phi_{0,a}=\phi_a$ for short. Define
\begin{align}
    F(a)=\fint_{\S} u\circ \phi_{-a}:=\frac{1}{2\pi}\int_{\S} u\circ \phi_{-a}.
\end{align}
One readily knows $F$ is continuous and maps $\D$ to $\D$ and $F(0)=\fint_{\S}u$.
For any $\zeta\in \partial \D=\S$, one has
\[\lim_{r\to 1^-}\phi_{-r\zeta}(\theta)=\zeta,\quad \forall\, \theta\neq -\zeta\]
and this convergence is uniform on $\{a\in\overline\D: |a+ \zeta|\geq \epsilon\}$ for any $\epsilon>0$. This implies that $\lim_{r\to 1^-}F(r\zeta)=u(\zeta)\in \S$ uniformly in $\zeta$. Hence $F$ is a continuous function on $\overline{\D}$. Moreover, the Lerary-Schauder degree of $F$ with respect to 0 is the same as the winding number of $u$ (see \cite{outerelo2009mapping}). By  Leray-Schauder degree theory, there exists $ a_0\in \D$ such that  $F(a_0)=0$.

In the general case of a map $u\in \mathcal{A}^1_{\S}$, one can approximate $u$ by a sequence of $u_j\in C^\infty(\S;\S)$ with the property that
\[u_j\to u \text{ strongly with } u\in \dot H^{1/2}(\S;\S)\quad \text{and }\quad \deg u_j=\deg u.\]

Going to a subsequence if necessary, we can assume $u_j\to u$ a.e. $\S$. By the previous paragraph, there exists $a_j\in \D$ such that
\[\fint_{\D} u_j\circ\phi_{a_j}=0.\]

We must have $|a_j|<1-\epsilon_0$. Suppose not, then there is a subsequence, which we still label it as $a_j$, converging to some  $-\zeta\in \partial\D$. Then $\phi_{a_j}\to \zeta$ a.e on $\S$  and $u_j\circ\phi_{a_j}\to u(\zeta)$ a.e. on $\S$.
So we can use Dominated Convergence Theorem to infer that
\[u(\zeta)=\fint_{\S}u(\zeta)=\lim_{j\to \infty}\fint_{\S}u_j\circ\phi_{a_j}=0.\]
However, since $|u_j\circ\phi_{a_j}|\equiv 1$,
\[|u(\zeta)|=\fint_{\S}|u(\zeta)|=\lim_{j\to \infty}\fint_{\S}|u_j\circ \phi_{a_j}|=1.\]
This is a contradiction. Now we can assume $a_j\to a_0$ where $a\in \D$ and consequently
\[\fint_{\S}u\circ \phi_{a_0}=0.\]
\end{proof}

\begin{proof}[{\bf Proof of Theorem \ref{intro:thm:sta-type1}}]
By the Lemma \ref{lem:mean=0},  it suffices to prove the stability \eqref{inf_d-D} for $u\in \mathcal{A}_{\S}^1$ with the additional assumption $\fint_{\S} u=0$. This is due to the invariance of energy and the degree under M\"obius transformation of $\D$, namely
\begin{align}
  \mathcal{E}(u)=\mathcal{E}( u\circ \phi_a),\quad \deg  u\circ \phi_a=\deg u=1
\end{align}
holds for any $a\in \D$ and the group structure of M\"obius transformations
\begin{align}
    \label{phiaphib}
    \phi_b\circ\phi_a=\phi_{\frac{a+b}{1+\bar ab}}.
\end{align}

If $D(u)\geq 1$, then
\begin{align}
   D(u)\geq 1=\deg u=\|u\|^2_{\cn}-D(u),
\end{align}
which means $\|u\|_{\cn}^2\leq 2D(u)$. Consequently \eqref{inf_d-D} holds. Thus for the following, we can assume $D(u)<1$.

Suppose that $u=\sum_{k\in \Z}c_kz^k$. By the previous step, we can assume $c_0=0$. Since $u\equiv 1$ a.e. $\S$, then
\begin{align*}
   D(u)=&\|u\|_{\cn}^2-\fint_{\S} u^2
   =\sum_{k\geq 1} [k-1](|c_k|^2+|c_{-k}|^2)\\
   \geq&\frac12\sum_{k\geq 2} k(|c_k|^2+|c_{-k}|^2)=\frac12\|u-c_1e^{\im\theta}-c_{-1}e^{-\im\theta}\|_{\cn}^2.
\end{align*}
That is
\begin{align}\label{u-c1-c_1}
    \|u-c_1e^{\im\theta}-c_{-1}e^{-\im\theta}\|_{\cn}^2\leq 2D(u).
\end{align}
Again $|u|\equiv 1$ a.e. $\S$ and \eqref{u=1-series} implies
\[\sum_{k\geq 1} |c_k|^2+|c_{-k}|^2=1.\]
Combining with the consequence of \eqref{deg-series}, that is
\[1=\deg u=\sum_{k\geq 1} k(|c_k|^2-|c_{-k}|^2),\]
one obtains
\begin{align}
    1-|c_1|^2\leq \sum_{k\geq 2} (|c_k|^2+|c_{-k}|^2)\leq D(u),\label{c1} \\
    |c_{-1}|^2\leq \frac12\sum_{k\geq 2} k(|c_k|^2-|c_{-k}|^2)\leq D(u).\label{c_1}
\end{align}
Therefore plugging in \eqref{c1} and \eqref{c_1} to \eqref{u-c1-c_1}, we obtain
\begin{align*}
    \|u-e^{\im(\theta+\arg c_1)}\|_{\cn}^2&\leq 9\left[2D(u)+\|c_{-1}e^{-\im\theta}\|_{\cn}^2+\|(e^{\im\arg c_1}-c_1)e^{\im\theta}\|^2_{\cn}\right]\\
    &\leq 9[3D(u)+(1-|c_1|)^2]\leq 9[3D(u)+D(u)^2].
\end{align*}
Since we assume $D(u)<1$, then the above inequality implies
\[\|u-e^{\im\vartheta}\id\|_{\cn}^2\leq 36D(u),\]
for some $\vartheta$. Therefore \eqref{inf_d-D} is established.

\end{proof}

Now we plan to prove Theorem \ref{intro:thm:local}. Before that, let us make some preparations.
First, we can expand $\phi_{\vartheta,\vec{b}}$ to be
\begin{align}
    \phi_{\vartheta,\vec{b}}=e^{\im \vartheta}\sum_{k\geq 0} A_k(\vec{b})z^k
\end{align}
where $A_k(\vec{b})$ is some function for $\vec{b}$ for each $k$. For instance, for any $a\in \D$,
\begin{align}
\begin{split}\label{phia-exp}
    \phi_a=\frac{z-a}{1-\bar a z}=(z-a)\sum_{k\geq0}\bar{a}^kz^k
    =-a+\sum_{k\geq 1}(1-|a|^2)\bar{a}^{k-1}z^k.
\end{split}
\end{align}
Since $\phi_{\vartheta,\vec{b}}=e^{\im\vartheta}\phi_{b_1}\cdots\phi_{b_d}$, one knows $A_0(\vec{b})=(-1)^d\prod_{j=1}^db_j$. For any vector $\vec{b}\in \D^d$ and a set $\Omega\in \D$, we abuse the notation that $\vec{b}\in \Omega^d$ if $b=(b_1,\cdots,b_d)$ with $b_i\in \Omega$ for $i=1,\cdots,d$.
\begin{lemma}\label{lem:A-lower-bd}
Suppose $\Omega\Subset \D$ is  compact. There exists a $\epsilon_\Omega>0$ and $N_\Omega$ such that, for any Blaschke product $\phi_{\vartheta,\vec{b}}$ with $\vec{b}\in \Omega^d$, we have
\begin{align}\label{A-lower-bd}
    \max_{1\leq i\leq N_\Omega}\{|A_i(\vec{b})|\}\geq \epsilon_\Omega.
\end{align}
\end{lemma}
\begin{proof}
Let $\epsilon(N)=\sum_{1\leq k\leq N}|A_k(\vec{b})|^2$. By the assumption, there exists $\epsilon_1=\epsilon_1(\Omega)$ such that
\begin{align}\nonumber
    |A_0(\vec{b})|^2\leq 1-\epsilon_1,\quad \forall \,\vec{b}\in\Omega^d.
\end{align}
By the property \eqref{u=1-series}, we get $\sum_{k\geq 0}|A_k(\vec{b})|^2=1$. Then
\begin{align}\nonumber
    \sum_{k\geq N+1}|A_k(\vec{b})|^2\geq \epsilon_1-\epsilon(N).
\end{align}
Since $\deg \phi_{\vartheta,\vec{b}}=d$ and \eqref{deg-series}, we get $\sum_{k\geq1 }k|A_k(\vec{b})|^2=d$. Then
\begin{align}\nonumber
    \sum_{k\geq N+1}k|A_k(\vec{b})|^2\leq d-\epsilon(N).
\end{align}
Compare the above two inequalities, we get
\begin{align}\nonumber
    d-\epsilon(N)>(N+1)(\epsilon_1-\epsilon(N)).
\end{align}
Now we choose $N_\Omega$ such that $(N_\Omega+1)\epsilon_1>d$. One readily see that
\begin{align}\nonumber
    \epsilon(N_\Omega)>\frac{1}{N_\Omega}[(N_\Omega+1)\epsilon_1-d].
\end{align}
This implies \eqref{A-lower-bd} holds for $\epsilon_\Omega=\left(\epsilon(N_\Omega)/N_\Omega\right)^{1/2}$.
\end{proof}

\begin{proposition}\label{prop:H1/2-L2}
Suppose $\Omega\Subset \D$ is  compact. There exists a constant $C_\Omega$ such that, for any $u\in \mathcal{A}_{\S}^d$ and $\vec{b}\in\Omega^d$, we have
\begin{align}
    \|u-\phi_{\vartheta,\vec{b}}\|_{L^2(\S)}^2\leq C_\Omega\left(\|u-\phi_{\vartheta,\vec{b}}\|_{\cn}^2+D(u)+D(u)^2\right).
\end{align}
\end{proposition}
\begin{proof}

Without loss of generality, we will assume $\vartheta=0$ in the assumption, and write $\phi_{\vartheta,\vec{b}}=\phi_{\vec{b}}$ for short. Suppose that $u\in \mathcal{A}_{\S}^d$ with the Fourier expansion $u=\sum_{k\in \Z}c_kz^k$.
By \eqref{2norm-series}, \eqref{deg-series} and \eqref{def:D(u)},
\begin{align}\label{D(u)-series}
    D(u)=\sum_{k\in \Z}|k||c_k|^2-\sum_{k\in \Z}k|c_k|^2=2\sum_{k\leq -1}|k||c_{k}|^2.
\end{align}
Now since $u-\phi_{\vec{b}}=\sum_{k\geq 0}(c_k-A_k)z^k+\sum_{k\leq -1}c_kz^k$, then
\begin{align}\label{kck-Ak}
    \|u-\phi_{\vec{b}}\|_{\cn}^2=\sum_{k\geq 1}k|c_k-A_k(\vec{b})|^2+\sum_{k\leq -1}|k||c_k|^2
\end{align}
and
\begin{align}
    \|u-\phi_{\vec{b}}\|_{L^2(\S)}^2=2\pi\sum_{k\geq 0}|c_k-A_k(\vec{b})|^2+2\pi\sum_{k\leq -1}|c_k|^2.
\end{align}
Combining \eqref{D(u)-series}, \eqref{kck-Ak}, we get
\begin{align}
    \|u-\phi_{\vec{b}}\|_{L^2(\S)}^2\leq 2\pi\left(|c_0-A_0(\vec{b})|^2 +\|u-\phi_{\vec{b}}\|_{\cn}^2+D(u)\right).
\end{align}
It suffices to establish the following claim.
\begin{claim}\label{claim:A0-c0}
There exists a constant $C_\Omega$ independent of $u$ and $\vec{b}$ such that
\begin{align}\label{diff-A0-c0}
|A_0(\vec{b})-c_0|\leq C_\Omega\left (\|u-\phi_{\vec{b}}\|_{\cn}+D(u)+\sqrt{D(u)}\right).
\end{align}
\end{claim}

In fact, Lemma \ref{lem:A-lower-bd} implies that there exists $1\leq k\leq N_\Omega$ such that $|A_k(\vec{b})|\geq \epsilon_\Omega$.
Without loss of generality, we assume $k=1$ and $|A_1(\vec{b})|\geq \epsilon_\Omega$. The following proof still works for other $k$ with minor modification.

Since $|u|=1$ and $|\phi_{\vec{b}}|=1$ a.e. on $\S$, \eqref{u=1-series} implies
\begin{align}
    \bar c_0c_1=-\sum_{j\in \Z\setminus\{0\}}\bar c_j c_{j+1},\quad \bar A_0 A_1=-\sum_{j\geq 1}\bar A_j A_{j+1}.
\end{align}
Here for the time being we write $A_k(\vec{b})=A_k$ for short.
Subtract two equations and make interpolation
\begin{align*}
    A_1(\bar c_0-\bar A_0)=&-\bar c_0(c_1-A_1)-\sum_{j\geq 1}[\bar c_{j}c_{j+1}-\bar A_{j}A_{j+1}]-\sum_{j\leq -1}\bar c_j c_{j+1}\\
    =&-\bar c_0(c_1-A_1)-\sum_{j\geq 1}[(\bar c_{j}-\bar A_{j})c_{j+1}+\bar  A_{j}(c_{j+1}-A_{j+1})]-\sum_{j\leq -1}\bar c_j c_{j+1}.
\end{align*}
Applying H\"older's inequality to get
\begin{align}
\begin{split}
    |A_1||c_0-A_0|\leq&\ \left(\sum_{j\geq 1}|c_j-A_{j}|^2\right)^{\frac12}\left(\sum_{j\geq 0}|c_{j}|^2+|A_{j}|^2\right)^{\frac12}+|D(u)|+|c_{-1}c_0|\\
    \leq&\  2\|u-\phi_{\vec{b}}\|_{\cn}+|D(u)|+\sqrt{D(u)}.
    \end{split}
\end{align}
By our assumption $|A_1|\geq \epsilon_{\Omega}$, the inequality \eqref{diff-A0-c0} holds.

\end{proof}
\begin{lemma}\label{lem:loc-zeros}
Suppose $\Omega\Subset \D$ is compact. For any $\varepsilon>0$, there exists $\delta_{\Omega,\varepsilon}>0$ such that if $u\in \mathcal{A}_{\S}^d$ and
\[\|u-\phi_{\vartheta,\vec{b}}\|_{\cn}^2\leq \delta_{\Omega,\epsilon}\]
for some $\vartheta\in \S$ and $\vec{b}\in \Omega^d$,  then the harmonic extension of $u$ has at least one zero at $\Omega_\varepsilon=\cup_{p\in \Omega}B_{\varepsilon}(p)$. Here $B_{\varepsilon}(p)$ denotes an open ball in $\D$.
\end{lemma}
\begin{proof}
Suppose $u=\sum_{k\in\mathbb{Z}}c_kz^k$. The harmonic extension $U$ of $u$ are simply $U(z)=\sum_{k\geq 0}c_kz^k+\sum_{k<0}c_k\bar z^k$ where $z\in \D$. Therefore
\begin{align}
    \|U-\phi_{\vartheta,\vec{b}}\|_{L^2(\D)}^2=\pi^2\sum_{k\geq 0}|c_k-A_k(\vec{b})|^2+\pi^2\sum_{k\leq -1}|c_k|^2=\frac{\pi}{2}\|u-\phi_{\vartheta,\vec{b}}\|_{L^2(\S)}^2.
\end{align}
Suppose $\varepsilon>0$ is small enough such that $\Omega_{\varepsilon}\Subset \D$. Otherwise, Leray-Schauder degree theory tells us the conclusion holds obviously.

Since $U-\phi_{\vartheta,\vec{b}}$ is a harmonic function in $\D$, then by the interior estimates
\begin{align}
    \|U-\phi_{\vartheta,\vec{b}}\|_{C^0(\overline{\Omega_{\varepsilon}})}\leq C_{\Omega,\varepsilon} \|U-\phi_{\vartheta,\vec{b}}\|_{L^2(\D)}.
\end{align}
By the property of Leray-Schauder degree, there exists $\eta=\eta(\Omega,\varepsilon)$ such that if $|U-\phi_{\vartheta,\vec{b}}|_{C^0(\partial \Omega_{\varepsilon})}<\eta$, then
\[\deg(U,\Omega_\varepsilon,0)=\deg (\phi_{\vartheta,\vec{b}},\Omega_\varepsilon,0).\]
Thus, we choose $\delta_{\Omega,\varepsilon}=2\eta^2 /(\pi  C_{\Omega,\varepsilon}^2)$. The above equation will be true for any $u$ satisfying $\|u-\phi_{\vartheta,\vec{b}}\|_{\cn}^2<\delta_{\Omega,\varepsilon}$.
However, the zeros of $\phi_{\vartheta,\vec{b}}$ are $b_1,b_2,\cdots,b_d$, which all contained in $\Omega$. Therefore $\deg(\phi_{\vartheta,\vec{b}},\Omega_{\varepsilon},0)=d$. The above identity means $u$ has at least one zero in $\Omega_{\varepsilon}$.
\end{proof}
Finally, we give the poof of the local version of quantitative stability.
\begin{proof}[{\bf Proof of Theorem \ref{intro:thm:local}}]
Here we work on the norm $\cn$, but one can easily translate results to norm $\cnr$ by \eqref{isometry}.
Without loss of generality, we will assume $\vartheta=0$ in the assumption, and write $\phi_{\vartheta,\vec{b}}=\phi_{\vec{b}}$ for short. Choose any $\varepsilon<\text{dist}(\Omega,\partial \D)/2$ and fix it for the rest of proof. Obviously $\Omega_{\varepsilon}\Subset \D$.
Suppose $u\in\mathcal{A}_{\S}^d$ and
\begin{align}\label{u-phi<delta}
    \|u-\phi_{\vec{b}}\|_{\cn}^2\leq \delta_{d,\Omega,\epsilon}
\end{align}
holds for some small $\delta_{d,\Omega,\epsilon}$, which is to be chosen later.
If $D(u)\geq \delta_{d,\Omega,\epsilon}$, then we can take $\vec{a}=\vec{b}$ and $\vartheta'=0$. Then
 \begin{align}\label{aDu}
    \|u-\phi_{\vartheta',\vec{a}}\|_{\cn}^2\leq C_{d,\Omega,\epsilon} D(u)
\end{align}
holds for $C_{d,\Omega,\epsilon}=1$. Therefore, we will assume $D(u)<\delta_{d,\Omega,\epsilon}$ for the rest of proof.


\begin{claim}\label{claim:reduce=0}
It suffices to prove the theorem for $u\in \mathcal{A}_{\S}^d$ with the additional assumption $\fint_{\S} u=0$.
\end{claim}

Indeed, it follows from Proposition \ref{prop:H1/2-L2} and Lemma \ref{lem:loc-zeros} that there exists $\tilde\delta_{d,\Omega,\varepsilon}$ such that if $\delta_{d,\Omega,\epsilon}<\tilde\delta_{d,\Omega,\varepsilon}$ then the harmonic extension $U$ of $u$ has a zero, say $a$, within $\Omega_{\varepsilon}$. Now consider  $u\circ\phi_{a}$. It satisfies $\int_{\S}u\circ \phi_a=0$, because $u\circ \phi_a=\sum{c_k}\phi_a^k$, and
\begin{align}
    \int_{\S}u\circ\phi_a=\sum_{k\geq0} c_k(-a)^k+\sum_{k\leq-1}c_k(-\bar{a})^k=U(-a)=0.
\end{align}
Since composing M\"obius transformation does not alter the $\cn$ norm,
\[\|u\circ\phi_a-\phi_{\vec{b}}\circ\phi_a\|_{\cn}^2\leq \delta_{d,\Omega,\epsilon}.\]
 Moreover $\phi_{\vec{b}}\circ\phi_{a}=\phi_{\vec{e}}$ with
\[e_i=\frac{b_i+a}{1+\bar ab_i}=\phi_{-a}(b_i),\quad i=1,\cdots,d.\]
Note that $a$ belongs to a compact set $\Omega_{\varepsilon}$, which just depend on $\Omega$. Then $\Omega\circ\phi_a:=\phi_{-a}(\Omega)$ is uniformly away from $\partial \D$, because for any $z\in\Omega$
\[1-|\phi_{-a}(z)|^2=1-\left|\frac{z+a}{1+\bar a z}\right|^2=\frac{(1-|a|^2)(1-|z|^2)}{|1+\bar a z|^2}\geq \frac{1}{6}\text{dist}(\Omega,\partial \D)^2.\]
Suppose for $\Omega\circ\phi_a$ and $u\circ\phi_a$, we have find $\delta_{d,\Omega\circ\phi_a,\tilde \varepsilon}$ such that there exists $\vec{a}'\in \Omega\circ\phi_a+B_{\tilde\varepsilon}(0)$  and $\vartheta'\in \S$
\[\|u\circ \phi_a-\phi_{\vartheta',\vec{a}'}\|_{\cn}^2\leq C_{d,\Omega\circ \phi_a,\tilde \varepsilon}D(u\circ\phi_a).\]
Here $\tilde \varepsilon$ is chosen to satisfy
\begin{align}\label{phiaOmega}
    \phi_a(\Omega\circ\phi_a+B_{\tilde\varepsilon}(0))\subset \Omega_\varepsilon,\quad\forall a\in \Omega_{\varepsilon}.
\end{align}
Take $\delta_{d,\Omega,\varepsilon}=\min\{\tilde\delta_{d,\Omega,\epsilon},\ \inf\{\delta_{d,\Omega\circ\phi_a,\tilde \varepsilon}:a\in \Omega_\varepsilon\}\}$ and $C_{d,\Omega,\varepsilon}=\sup\{C_{d,\Omega\circ\phi_a,\tilde \varepsilon}:a\in \Omega_\varepsilon\}$. Again since $\Omega\circ\phi_a$ is uniformly away from $\partial \D$, one must have $\delta_{d,\Omega,\varepsilon}>0$ and $C_{d,\Omega,\varepsilon}<\infty$. Then for $\delta_{d,\Omega,\varepsilon}$, we can find $\vec{a}$ such that \eqref{aDu} holds with $C_{d,\Omega,\varepsilon}$. Moreover, \eqref{phiaOmega} means that $\vec{a}$ can be chosen in $\Omega_\varepsilon$. Therefore Claim \ref{claim:reduce=0} is proved.

With the above claim in hand, we will prove the statement by induction.

Suppose $d=1$, we write $\vec{b}=b_1$ in \eqref{u-phi<delta}. Theorem \ref{intro:thm:sta-type1} guarantee the existence of $\vec{a}=a_1$ such that \eqref{aDu} holds with $C_{1,\Omega,\epsilon}=36$. It suffices to show $|a_1-b_1|<\varepsilon$ if $\delta_{1,\Omega,\varepsilon}$ is chosen small enough. If fact, if \eqref{u-phi<delta} and \eqref{aDu} holds, then
\begin{align}
    \|\phi_{\vartheta,b_1}-\phi_{\vartheta',a_1}\|_{\cn}^2\leq 2\|u-\phi_{\vartheta,b_1}\|_{\cn}^2+2\|u-\phi_{\vartheta',a_1}\|_{\cn}^2\leq 100\delta_{1,\Omega,\epsilon}.
\end{align}
Direct computation shows that
\begin{align*}
    \|\phi_{\vartheta,b_1}-&\phi_{\vartheta',a_1}\|_{\cn}^2= \|\phi_{\vartheta-\vartheta',\frac{b_1-a_1}{1-\bar a_1b_1}}-\id\|_{\cn}^2\\
    =& \left|e^{\im (\vartheta-\vartheta')}\left(1-\left|\frac{b_1-a_1}{1-\bar a_1 b_1}\right|\right)-1\right|^2+\left|\frac{b_1-a_1}{1-\bar a_1b_1}\right|^2\left(2-\left|\frac{b_1-a_1}{1-\bar a_1b_1}\right|^2\right).
\end{align*}
Now it is easy to deduce that
\[|b_1-a_1|\leq 20\sqrt{\delta_{1,\Omega,\epsilon}}.\]
Choose $\delta_{1,\Omega,\varepsilon}$ small so that $|a_1-b_1|<\varepsilon$. This proves the case $d=1$. Now by induction, we can assume the theorem holds for the case $d-1$ with $\delta_{d-1,\Omega,\varepsilon}$ and $C_{d-1,\Omega,\varepsilon}$.

\begin{claim}
There exists $\hat \delta_{d-1,\Omega,\varepsilon}$ such that if $\delta_{d,\Omega,\epsilon}<\hat \delta_{d-1,\Omega,\varepsilon}$, then one can find $\vec{b}'\in \Omega^{d-1}$ such that
\begin{align}\label{aomega'}
    \|u\bar z-\phi_{\vec{b}'}\|^2_{\cn}<\delta_{d-1,\Omega,\varepsilon}.
\end{align}
\end{claim}

To prove the claim, we notice that $\phi_{\vec{b}}$ is invariant under the permutation of $b_i$. Without loss of generality, we assume $|b_1|\leq \cdots\leq |b_d|$\footnote{such ordering might not be unique, but it does not affect the proof}. Recalling $|A_0(\vec{b})|=\prod_{1\leq k\leq d}|b_k|$ and \eqref{diff-A0-c0}, we get $|b_1|\leq 10 \delta_{d,\Omega,\epsilon}^{1/4d}$. Let $\vec{b}'=(b_2,\cdots,b_d)\in \Omega^{d-1}$. By the Cauchy inequality
\begin{align}
    \|u\bar z-\phi_{\vec{b}'}\|_{\cn}^2\leq 2\|u\bar z-\bar z\phi_{\vec{b}}\|_{\cn}^2+2\|\bar z\phi_{\vec{b}}-\phi_{\vec{b}'}\|_{\cn}^2.
\end{align}
One can compute that
\begin{align*}
     \|u\bar z-\bar z\phi_{\vec{b}}\|_{\cn}^2-\|u-\phi_{\vec{b}}\|^2_{\cn}&=\sum|k||c_{k+1}-A_{k+1}(\vec{b})|^2-\sum|k||c_k-A_k(\vec{b})|^2\\
     &=-\sum_{k\geq 1}|c_k-A_k(\vec{b})|^2+\sum_{k\leq 0}|c_k|^2<D(u).
\end{align*}
Notice that
\begin{align}
    \|\bar z\phi_{\vec{b}}-\phi_{\vec{b}'}\|_{\cn}^2=\|\frac{1-b_1\bar z}{1-\bar b_1 z}\phi_{b_2}\cdots\phi_{b_d}-\phi_{b_2}\cdots\phi_{b_d}\|_{\cn}^2\to 0
\end{align}
as $\delta_{d,\Omega,\epsilon}\to 0$. Therefore, choosing $\delta_{d,\Omega,\epsilon}$ small enough, we can get \eqref{aomega'}.

Therefore, by induction, we have $\vartheta'\in \S$ and $\vec{a}'\in \Omega^{d-1}_{\varepsilon/2}$ such that
\begin{align}\label{u-21}
    \|u\bar z-\phi_{\vartheta',\vec{a}'}\|_{\cn}^2\leq C_{d-1,\Omega_{\varepsilon/2},\varepsilon}D(u\bar z).
\end{align}
\begin{claim}\label{claim:u-z-d-1}
We claim that there exists $C_{d,\Omega,\varepsilon}$ such that
\begin{align}\label{u-z-C1}
    \|u-z\phi_{\vartheta',\vec{a}'}\|_{\cn}^2\leq C_{d,\Omega,\varepsilon}D(u).
\end{align}
\end{claim}

Indeed, since $u\bar z=\sum c_{k+1}z^k$,
\begin{align}
\begin{split}\label{Dubarz}
    D(u\bar z)=&2\sum_{k\leq -1}|k||c_{k+1}|^2=2\sum_{k\leq -1}|k+1||c_{k+1}|^2+2\sum_{k\leq -1}|c_{k+1}|^2\\
    =&D(u)+2\sum_{k\leq 0}|c_{k}|^2\leq 2D(u).
\end{split}
\end{align}
Here in the last step, we have used $c_0=0$.  Then \eqref{u-21} and \eqref{Dubarz} imply  \begin{align}\label{ubarz-Du}
    \|u\bar z-\phi_{\vartheta',\vec{a}'}\|_{\cn}^2\leq 2C_{d-1,\Omega_{\varepsilon/2},\varepsilon}D(u).
\end{align}
Since $u-z\phi_{\vartheta',\vec{a}'}=\sum(c_k-e^{\im \vartheta'}A_{k-1}(\vec{a}'))z^k$, then
\begin{align*}
\begin{split}
    \|u-z\phi_{\vartheta',\vec{a}'}\|^2_{\cn}=&\sum |k||c_k-e^{\im \vartheta'}A_{k-1}(\vec{a}')|^2=\sum|k+1||c_{k+1}-e^{\im \vartheta'}A_k(\vec{a}')|^2\\
    =&\|u\bar z-\phi_{\vartheta',\vec{a}'} \|_{\cn}^2+\sum_{k\geq 0}|c_{k+1}-e^{\im \vartheta'}A_k(\vec{a}')|^2-\sum_{k\leq 0}|c_{k}|^2.
    \end{split}
\end{align*}
Since \eqref{2norm-series}, we have $\sum_{k\geq 1}k|c_{k+1}-e^{\im \vartheta'}A_k(\vec{a}')|^2\leq\|u\bar z-\phi_{\vartheta',\vec{a}'}\|^2_{\cn} $
\begin{align}
\begin{split}\label{u-22}
    \|u-z\phi_{\vartheta',\vec{a}'}\|^2_{\cn}\leq&2\|u\bar z-\phi_{\vartheta',\vec{a}'} \|^2_{\cn}+|c_1-e^{\im\vartheta'}A_0(\vec{a})|^2\\
    \leq& 4C_{d-1,\Omega_{\varepsilon/2},\varepsilon}D(u)+|c_1-e^{\im\vartheta'}A_0(\vec{a})|^2.
\end{split}
\end{align}
Since $|u|=1$ and $|\phi_{\vartheta',\vec{a}'}|=1$ a.e. on $\S$, \eqref{u=1-series} implies
\begin{align}
    \bar c_1c_2=-\sum_{j\in \Z\setminus\{1\}}\bar c_j c_{j+1},\quad \bar A_0 A_1=-\sum_{j\geq 1}\bar A_j A_{j+1}.
\end{align}
Here and the following we write $A_k(\vec{a}')=A_k$ for short. We subtract two equations and make interpolations.
\begin{align*}
    e^{\im \vartheta'}A_1(\bar c_1-e^{-\im \vartheta'}\bar A_0)
    =&-\bar c_1(c_2-e^{\im \vartheta'}A_1)-\sum_{j\geq 2}[\bar c_{j}c_{j+1}-\bar A_{j-1}A_{j}]-\sum_{j\leq -2}\bar c_j c_{j+1}\\
    =&-\bar c_1(c_2-e^{\im \vartheta'}A_1)-\sum_{j\leq -2}\bar c_j c_{j+1}\\
    &-\sum_{j\geq 2}[(\bar c_{j}-e^{-\im \vartheta'}\bar A_{j-1})c_{j+1}+e^{-\im\vartheta'}\bar  A_{j-1}(c_{j+1}-e^{\im \vartheta'}A_{j})].
\end{align*}
Applying H\"older's inequality and \eqref{ubarz-Du}, we have
\begin{align}\label{ineq-small}
\begin{split}
    |A_1||c_1-e^{\im \vartheta'}A_0|\leq&\ \left(\sum_{j\geq 2}|c_j-e^{\im \vartheta'}A_{j-1}|^2\right)^{\frac12}\left(\sum_{j\geq 1}|c_{j}|^2+|A_{j}|^2\right)^{\frac12}+D(u)\\
    \leq&\ 2\|u\bar z-\phi_{\vartheta',\vec{a}'}\|_{\cn}+D(u)
    \leq\ 3\sqrt{2C_{d-1,\Omega_{\varepsilon/2},\varepsilon}D(u)}.
    \end{split}
\end{align}
Here we used $D(u)<\delta <1$. Using Lemma \ref{lem:A-lower-bd}, we have
\begin{align}
    |c_1-A_0|\leq 3\epsilon_{\Omega_{\varepsilon}}^{-1}\sqrt{2C_{d-1,\Omega_{\varepsilon/2},\varepsilon}D(u)}.
\end{align}
Plugging this back to \eqref{u-22} to get
\begin{align}
    \|u-z\phi_{\vartheta',\vec{a}}\|^2_{\cn}\leq C_{d,\Omega,\varepsilon}D(u)
\end{align}
with $C_{d,\Omega,\varepsilon}=(4+18\epsilon_{\Omega_\varepsilon}^{-2})C_{d-1,\Omega_{\varepsilon/2},\varepsilon}$.
The \eqref{u-z-C1} is proved.

Having established Claim \ref{claim:u-z-d-1}, we can see \eqref{aDu} holds with the choice of  $\vec{a}=(0,\vec{a}')\in \Omega_{\varepsilon}^{d}$, $\delta_{d,\Omega,\varepsilon}=\min\{\tilde\delta_{d,\Omega,\varepsilon},\hat\delta_{d-1,\Omega,\varepsilon}\}$ and $C_{\Omega,d,\varepsilon}=\max\{(4+18\epsilon_{\Omega_\varepsilon}^{-2})C_{d-1,\Omega_{\varepsilon/2},\varepsilon},1\}$.  The induction is complete.
\end{proof}



\section{A counter example for higher degree}\label{sec:higher}

In this section we shall prove that there is no uniform stability. Recall that all the half-harmonic maps from $\R$ to $\S$ with positive degree can be written in \eqref{intro:1/2-H-form}. Within this section, we will assume $\alpha_i=x_i+\im \lambda_i$ with $x_i\in \R$ and $\l_i>0$. Then
\begin{align}
    \psi_{\vartheta,\vec{\alpha}}=e^{\im\vartheta}\frac{x-x_1+\im \l_1}{x-x_1+\im \l_1}\frac{x-x_2-\im \l_2}{x-x_2+\im \l_2}.
\end{align}


 One can equip $\dot H^{1/2}(\R,\R^2)$ with the inner product,
\begin{align}
    \langle \bphi_1,\bphi_2\rangle=\frac{1}{2\pi}\iint_{\R\times \R}\frac{(\bphi_1(x)-\bphi_1(y))\cdot(\bphi_2(x)-\bphi_2(y))}{|x-y|^2}dxdy.
\end{align}
If $\bphi_2$ is smooth enough, then RHS of the above equation can be written  $\int_{\R}\bphi_1\cdot \DRh\bphi_2$. We shall abuse the notation by denoting $\langle \bphi_1,\bphi_2\rangle =\int_{\R}\bphi_1\cdot \DRh\bphi_2$ for all $\bphi_1,\bphi_2\in \dot H^{1/2}(\R,\R^2)$. It should be interpreted as the RHS of the above equation. Let us use $\bpsi^\perp$ denote the vector rotating $\bpsi$  counterclockwise by $\pi/2$.
If $\bphi_1=h_1 \bpsi_{\vartheta,\vec{\alpha}}^\perp$  and $\bphi_2=h_2 \bpsi_{\vartheta,\vec{\alpha}}^\perp$ with $h_1,h_2\in H^{1/2}(\R;\R)$, then
\begin{align}\label{inner-prod}
    \begin{split}
         \langle\bphi_1,\bphi_2\rangle=&\ \int_{\R}h_1 \bpsi_{\vartheta,\vec{\alpha}}^\perp\cdot \DRh (h_2\bpsi_{\vartheta,\vec{\alpha}}^\perp)\\
    =&\  \int_{\R}h_1\DRh h_2+\frac{1}{2\pi}\iint_{\R\times \R}Q_{\vartheta,\vec{\alpha}}(x,y)h_1(x)h_2(y)dxdy,
    \end{split}
\end{align}
where
\begin{align}
    Q_{\vartheta,\vec{\alpha}}(x,y)=\frac{|\bpsi_{\vartheta,\vec{\alpha}}(x)-\bpsi_{\vartheta,\vec{\alpha}}(y)|^2}{|x-y|^2}.
\end{align}
\begin{remark}
For instance,
in the very special case $\psi=\frac{x-\l\im}{x+\l\im}$, then
\[\psi(x)-\psi(y)=\frac{2\l^2(x^2-y^2)}{(\l^2+x^2)(\l^2+y^2)}+\im\frac{2\l(y-x)(\l^2-xy)}{(\l^2+x^2)(\l^2+y^2)}.\]
Then
\begin{align}
    Q(x,y)=\frac{|\psi(x)-\psi(y)|^2}{|x-y|^2}=\frac{4\l^2}{(\l^2+x^2)(\l^2+y^2)}.
\end{align}
It follows from the non-degeneracy result that $\ker \mathcal{L}=\textup{span}\{K_1,K_2,K_3\}$ where
\[K_1=1,\quad K_2=\frac{x^2-\l^2}{x^2+\l^2},\quad K_3=\frac{2\l x}{x^2+1}.\]
Then it is easy to verify that
\[\langle K_1\bpsi^\perp,K_2\bpsi^\perp\rangle=\langle K_1\bpsi^\perp,K_3\bpsi^\perp\rangle=\langle K_2\bpsi^\perp,K_3\bpsi^\perp\rangle=0. \]
\end{remark}

Getting back to degree two case, we will mainly work on the case  $\vartheta=0$, $\alpha_1=j^2+\im $ and $\alpha_2=-j^2+\im $ in this section. For brevity, such
 $\psi_{\vartheta,\vec{\alpha}}$ will be written $\psi_j$ for short.
Denote
\begin{align}\label{rhoj-Qj}
    \rho_j(x)= \bpsi_j\cdot \DRh\bpsi_j,\quad  Q_j(x,y)&=\frac{|\bpsi_j(x)-\bpsi_j(y)|^2}{|x-y|^2}.
\end{align}
\begin{lemma}\label{lem:rhoQ}One can compute that
\begin{align}\label{rhoj}
   \rho_j(x)=\frac{4(1+j^4+x^2)}{[1+(x-j^2)^2][1+(x+j^2)^2]},
\end{align}
\begin{align}\label{R-ker}
    \begin{split}
       Q_j(x,y)=\frac{16(1+xy+j^4)^2}{[1+(x-j^2)^2][1+(x+j^2)^2][1+(y-j^2)^2][1+(y+j^2)^2]}.
    \end{split}
\end{align}
We always have
\begin{align}\label{int-rho=Q}
    \int_{\R}\rho_j=\frac{1}{2\pi}\iint_{\R\times\R}Q_j=4\pi.
\end{align}
\end{lemma}
\begin{proof}
 By a straightforward computation, one can get \eqref{rhoj} and \eqref{R-ker}. Note that $\|\psi_j\|_{\cnr}^2=|\deg \psi_j|=2$, then we have
 \begin{align}
      \int_{\R}\rho_j=\frac{1}{2\pi}\iint_{\R\times\R}Q_j=2\pi\|\psi_j\|_{\cnr}^2=4\pi.
 \end{align}
\end{proof}

 Introduce the notation
\begin{align}
\begin{split}\label{def:Ki}
    K_{1}(x)=1,\quad K_{2}(x)=\frac{2}{(x-j^2)^2+1},\quad K_3(x)=\frac{2(x-j^2)}{(x-j^2)^2+1},\\
    K_4(x)=\frac{2}{(x+j^2)^2+1},\quad K_5(x)=\frac{2(x+j^2)}{(x+j^2)^2+1}.
    \end{split}
\end{align}
It is easy to know $\ker \mathcal{L}_{\bpsi_j}=\text{span}\{K_1,K_2,K_3,K_4,K_5\}$.

Define $\mathcal{J}=(J_{kl})_{1\leq k,l\leq 5}$ with
\begin{align}
    J_{kl}=\langle K_k\bpsi_j^\perp,K_l\bpsi_j^\perp\rangle.
\end{align}

\begin{lemma}\label{lem:K_i-prod}
For any $j>0$, one can show
\begin{align*}
    J_{11}=4\pi,\quad J_{12}=\frac{2(2+j^4)\pi}{j^4+1}=J_{14},\quad J_{13}=\frac{-2j^2\pi}{j^4+1}=-J_{15},
\end{align*}
\begin{align*}
     J_{22}=\left(1+\frac{2j^8+5j^4 +5}{(j^4+1)^2}\right)\pi,\quad J_{23}=\frac{-2j^2}{(j^4+1)^2}\pi,\quad J_{24}=\frac{2j^4+6 }{(j^4+1 )^2}\pi,
\end{align*}
\[J_{25}=\frac{2j^2(j^4+3)}{(j^4+1 )^2}\pi,\quad J_{33}=\left(1+\frac{3j^4+1 }{(j^4+1 )^2}\right)\pi=J_{55},\quad J_{34}= -\frac{2j^2(j^4+3)}{(1 +j^4)^2}\pi,\]
\[J_{35}=\frac{2 -2j^4}{(j^4+1 )^2}\pi,\quad
    J_{44}=J_{22},\quad J_{45}=\frac{2j^2\pi}{(j^4+1 )^2}.\]
The determinant of $\mathcal{J}=(J_{kl})_{1\leq k,l\leq 5}$ is
\begin{align}
    \begin{split}
       \det\mathcal{J}&=\frac{j^8(3j^{16}+22j^{12}+51 j^8+48j^4+16)}{(j^4+1 )^6}\pi^5.
    \end{split}
\end{align}
\end{lemma}
\begin{proof}
We shall use  \eqref{inner-prod} to compute all the inner products. First, let us compute $\DRh K_i$ for $i=1,\cdots, 5$.
It is easy to know $\DRh  K_1=0$. By the extension method, we can obtain
\begin{align*}
    \DRh K_{2}&=-\frac{\partial}{\partial y}\Big|_{y=0}\frac{2y+2}{(x-j^2)^2+(y+1)^2}=\frac{2(1-(x-j^2)^2)}{(1+(x-j^2)^2)^2},\\
    \DRh K_{3}&=-\frac{\partial}{\partial y}\Big|_{y=0}\frac{ 2(x-j^2)}{(x-j^2)^2+(y+1)^2}=\frac{4(x-j^2)}{(1+(x-j^2)^2)^2},
\end{align*}
\begin{align*}
    \DRh K_{4}&=-\frac{\partial}{\partial y}\Big|_{y=0}\frac{2y+2}{(x+j^2)^2+(y+1)^2}=\frac{2(1-(x+j^2)^2)}{(1+(x+j^2)^2)^2},\\
    \DRh K_{5}&=-\frac{\partial}{\partial y}\Big|_{y=0}\frac{ 2(x+j^2)}{(x+j^2)^2+(y+1)^2}=\frac{4(x+j^2)}{(1+(x+j^2)^2)^2}.
\end{align*}
Using \eqref{inner-prod} and \eqref{R-ker}, we can compute each $J_{kl}=\langle K_k\bpsi_j^\perp,K_l\bpsi_j^\perp\rangle$ respectively. Since the integrals here involve only rational functions, we take advantage of a symbolic software, Mathematica, to aid our computation.
\end{proof}
\begin{proposition}\label{prop:module}
Fix any $j\geq 100$, there exists $\delta_j$  such that for any $u$ with $\|\u-\bpsi_{j}\|_{\dot H^{1/2}(\R) }<\delta_j$, then there exists a unique $\vartheta=\vartheta(\u)$, $\vec{\alpha}=\vec{\alpha}(\u)$ satisfying
\begin{align}\label{ucdotg}
    \int_{\R}\u\cdot \DRh(g\bpsi^\perp_{\vartheta,\vec{\alpha}})=0,\quad \forall\, g\in \ker \mathcal{L}_{\bpsi_{\vartheta,\vec{\alpha}}}.
\end{align}

\end{proposition}
\begin{proof}

 Define the following function
 \begin{align*}
 \Phi:\dot H^{1/2}(\R)\times &\S\times \mathbb{H}^2\to \R^5
 \end{align*}
 \begin{align*}
     (\u,\vartheta,\vec{\alpha})&\mapsto (\langle \u,\K^1\rangle,\langle \u,\K^2\rangle,\langle \u,\K^3\rangle,\langle \u,\K^4\rangle,\langle \u,\K^5\rangle)
 \end{align*}
 where
$\K^i$ are obtained from $\ker L_{\psi_{\vartheta,\vec{\alpha}}}=\text{span}\{\K^1,\K^2,\K^3,\K^4,\K^5\}$. More precisely
\begin{align}\label{def:calKi}
\begin{split}
     \K^1=\bpsi_{\vartheta,\vec{\alpha}}^\perp, \quad
     \K^2=\frac{2\lambda_1^2}{(x-x_1)^2+\lambda_1^2}\bpsi_{\vartheta,\vec{\alpha}}^\perp,\quad \K^3=\frac{2\lambda_1(x-x_1)}{(x-x_1)^2+\lambda_1^2}\bpsi_{\vartheta,\vec{\alpha}}^\perp,\\
     \K^4=\frac{2\lambda_2^2}{(x-x_2)^2+\lambda_2^2}\bpsi_{\vartheta,\vec{\alpha}}^\perp,\quad \K^5=\frac{2\lambda_2(x-x_2)}{(x-x_2)^2+\lambda_2^2}\bpsi_{\vartheta,\vec{\alpha}}^\perp.
\end{split}
\end{align}
Such $\Phi$ is well-defined, because $\u$ and $\K^i$ all belong to $\dot H^{1/2}(\R;\R^2)$. Moreover, since  $\K^i$ depends on $\vartheta,\vec{\alpha}$ smoothly, then $\Phi$ also depends on $\vartheta,\vec{\alpha}$ smoothly. Since the inner product \eqref{inner-prod} depends on its arguments smoothly, consequently $\Phi$ depends on $\u$ smoothly. Moreover, since $\int_{\R}\bpsi_{\vartheta,\vec{\alpha}}\cdot\DRh\K^i=\int_{\R}\DRh\bpsi_{\vartheta,\vec{\alpha}}\cdot\K^i=0$ for any $\vartheta,\vec{\alpha}$ and $i=1,\cdots,5$, then
\begin{align}\label{Phi-v}
    \Phi(\u,\vartheta,\vec{\alpha})=\Phi(\u-\bpsi_{\vartheta,\vec{\alpha}},\vartheta,\vec{\alpha}):=\Phi(\v_{\vartheta,\vec{\alpha}},\vartheta,\vec{\alpha}).
\end{align}
Here we introduced the notation $\v_{\vartheta,\vec{\alpha}}=\u-\bpsi_{\vartheta,\vec{\alpha}}$.

We intend to apply implicit function theorem to $\Phi$ at $(\bpsi_j,0,(j^2+\im,-j^2+\im))$. The Jacobian matrix with respect to parameters $\vartheta, \vec{\alpha}$ at $(\bpsi_j,0,(j^2+\im,-j^2+\im))$ is
\begin{align}
    &Jac_\Phi(\bpsi_j,0,(j^2+\im,-j^2+\im))=\ (\partial_{\vartheta}\Phi,\partial_{x_1}\Phi,\partial_{\l_1}\Phi,\partial_{x_2}\Phi,\partial_{\l_2}\Phi)^T.
\end{align}
If $\vartheta=0$ and $\vec{\alpha}=(j^2+\im,-j^2+\im)$, we will write $\v_j=\u-\bpsi_{\vartheta,\vec{\alpha}}$ and $\boldsymbol{K}^i_j=\K^i$ for short.
At $(\bpsi_j,0,(j^2+\im,-j^2+\im))$, one has $\v_j=0$ and $\boldsymbol{K}^i_j=K_i\bpsi_j^\perp$, where $K_i$ are defined in \eqref{def:Ki}. Therefore using \eqref{Phi-v}, we have
\begin{align*}
    \partial_{\vartheta}\Phi=(\langle \partial_\vartheta \v_j,K_1\bpsi_j^\perp\rangle,\langle \partial_\vartheta \v_j,K_2\bpsi_j^\perp\rangle,\langle \partial_\vartheta \v_j,K_3\bpsi_j^\perp\rangle,\langle \partial_\vartheta \v_j,K_4\bpsi_j^\perp\rangle, \langle \partial_\vartheta \v_j,K_5\bpsi_j^\perp\rangle ).
\end{align*}
Similar equality holds for $\partial_{x_1}\Phi$, $\partial_{\l_1}\Phi$, $\partial_{x_2}\Phi$ and $\partial_{\l_2}\Phi$.
Furthermore, one can compute that $\partial_\vartheta \v_j=-K_1\bpsi_j^\perp$ and
\begin{align}
    \begin{split}
         \partial_{x_1}\v_j\ =\frac{-2}{(x-j^2)^2+1}\bpsi_j^\perp=-K_2\bpsi_j^\perp,\quad
    \partial_{\l_1}\v_j\ =\frac{-2(x-j^2)}{(x-j^2)^2+1}\bpsi_j^\perp=-K_3\bpsi_j^\perp,\\
    \partial_{x_2}\v_j\ =\frac{-2}{(x+j^2)^2+1}\bpsi_j^\perp=-K_4\bpsi_j^\perp,\quad
    \partial_{\l_2}\v_j\ =\frac{-2(x+j^2)}{(x+j^2)^2+1}\bpsi_j^\perp=-K_5\bpsi_j^\perp.
    \end{split}
\end{align}
 Plugging in these computations to the Jacobian matrix and using Lemma \ref{lem:K_i-prod}, we have
\begin{align}
    \det Jac_\Phi(\bpsi_j,0,(j^2+\im,-j^2+\im))=-\det \mathcal{J}= -3\pi^5 +O(1/j^4)\neq 0.
\end{align}

The implicit function theorem gives that there exists $\delta_j>0$ and unique smooth functions $\vartheta(\u),x_1(\u),x_2(\u),\l_1(\u),\l_2(\u)$  such that for any $\u$ with $\|\u-\bpsi_j\|_{\dot H^{1/2}(\R)}<\delta_j$ one has $\Phi(\u,\vartheta,\vec{\alpha})=0$, where $\vec{\alpha}=(\alpha_1,\alpha_2)$, $\alpha_1=x_1+\im \l_1$ and $\alpha_2=x_2+\im \l_2$.
That is
\begin{align}
    \int_{\R}\u\cdot \DRh(g \bpsi_{\vartheta,\vec{\alpha}}^\perp)=0,\quad \forall\, g\in \ker \mathcal{L}_{\bpsi_{\vartheta,\vec{\alpha}}}.
\end{align}




\end{proof}

We recall a function defined in \cite{brezis1995degree}. For any $j\in \mathbb{N}$, define $f_j$ on $\R$ as
\begin{align}\label{fj}
    f_j(x)=\begin{cases}
    1&\text{if }|x|\leq j,\\
    2-\frac{\log|x|}{\log j}&\text{if } j\leq |x|\leq j^2,\\
    0&\text{if }|x|\geq j^2.
    \end{cases}
\end{align}
\begin{lemma}\label{lem:1/2norm-fj}
There exists some uniform constant $C$ such that
$\|f_j\|_{\dot H^{1/2}}^2\leq C/|\log j|$ for all $j\in \mathbb{N}$.
\end{lemma}
\begin{proof}
We extend $f_j$ to $\Rp$ by replacing $|x|$ with $|(x,y)|$. Since
\begin{align}
    \partial_xf_j=\begin{cases}
    \frac{x}{x^2+y^2}\frac{1}{\log j}&\text{if }j\leq \sqrt{x^2+y^2}\leq j^2,\\
    0&\text{otherwise. }
    \end{cases}
\end{align}
One has a similar expression for $\partial_y f_j$. Then
\begin{align}
    \int_{\Rp}\left(|\partial_xf_j|^2+|\partial_{y}f_j|^2\right)=\frac{1}{\log^2j}\int_{B_{j^2}\setminus B_{j}}\frac{1}{x^2+y^2}dxdy\leq \frac{C}{\log j}.
\end{align}
Therefore by \eqref{norm-ext-disk}, we get the conclusion.
\end{proof}

For any $j\geq100$ we define
\begin{align}\label{def:h}
    h(x)=f_{j}(x-j^2)-f_j(x+j^2).
\end{align}
Then $h$ is an odd function. Moreover, $\|h\|_{\dot H^{1/2}(\R)}\leq C/\log j$.

\begin{lemma}\label{lem:rhohQh} For any $j\geq100$, we have
\begin{align}\label{int_rhojh}
\int_{\R} \rho_j h^2=4\pi +O(1/j),
\end{align}
and
\begin{align}\label{int_Qjh}
   \frac{1}{2\pi} \iint_{\R\times \R} Q_j(x,y)h(x)h(y)dxdy=4\pi+O(1/j).
\end{align}
Consequently we get
\begin{align}\label{int_hLh-decay}
\int_{\R} h\mathcal{L}_{\bpsi_j}h=O(1/\log j),\quad
\langle h\bpsi_j^\perp,h\bpsi_j^\perp\rangle= 4\pi+O(1/\log j).
\end{align}
\end{lemma}
\begin{lemma}\label{lem:hK_i}
For any $j\geq100$, we have $\langle h\bpsi_j^\perp,K_1\bpsi_j^\perp\rangle=0$,
\begin{align}
\begin{split}
    \langle h\bpsi_j^\perp,K_2\bpsi_j^\perp\rangle=\pi+O(1/j),\quad\langle h\bpsi_j^\perp,K_3\bpsi_j^\perp\rangle=O(1/j^2),\\
    \langle h\bpsi_j^\perp,K_4\bpsi_j^\perp\rangle=-\pi+O(1/j),\quad\langle h\bpsi_j^\perp,K_5\bpsi_j^\perp\rangle=O(1/j^2).
    \end{split}
\end{align}
\end{lemma}

In order not to interrupt the main thread of this section, We will defer the proof of these two lemmas to the end of this paper.

\begin{lemma}\label{lem:hp-est}
There exists $\varepsilon_j$ with the following significance. For any $\varepsilon<\varepsilon_j$, there exists $h_\perp$ such that $\u=\varepsilon h_\perp\bpsi_j^\perp+\sqrt{1-\varepsilon^2h_\perp^2}\bpsi_j$ satisfies
\begin{align}\label{ortho-1}
    \int_{\R}\u\cdot\DRh(K_i\bpsi_j^\perp)=0,\quad i=1,2,3,4,5.
\end{align}
Furthermore,
\begin{align}
    \int_{\R}h_\perp \mathcal{L}_{\bpsi_j}h_\perp=O(1/\log j)+O(\varepsilon),\label{hpLhp}\\
    \langle h_\perp\bpsi_j^\perp,h_\perp\bpsi_j^\perp\rangle =\frac{10}{3}\pi+O(1/\log j)+O(\varepsilon). \label{<hp,hp>}
\end{align}
\end{lemma}
\begin{proof}
We can take
\begin{align}
    h_\perp=h-c_1K_1-c_2K_2-c_3K_3-c_4K_4-c_5K_5
\end{align}
with $c_i$ to be determined. Define a map
\begin{align}
    \begin{split}
       \Phi:\R_+\times \R^5\to& \ \R^5\\
    (\varepsilon,\vec{c})\mapsto& (\langle \v,K_1\bpsi_j^\perp\rangle,\langle \v,K_2\bpsi_j^\perp\rangle,\langle \v,K_3\bpsi_j^\perp\rangle,\langle \v,K_4\bpsi_j^\perp\rangle,\langle \v,K_5\bpsi_j^\perp\rangle)
    \end{split}
\end{align}
where $\vec{c}=(c_1,c_2,c_3,c_4,c_5)$ and
\begin{align}
    \v=h_\perp \bpsi_j^\perp-\frac{\varepsilon h_\perp^2}{\sqrt{1-\varepsilon^2h_\perp^2}+1}\bpsi_j.
\end{align}
The map $\Phi$ is well defined because $\v$ and $K_i\bpsi_j^\perp$ all belong to $\dot H^{1/2}(\R)$. At $\varepsilon=0$, $\Phi(0,\vec{c})=0$ if and only if
\begin{align}
    \mathcal{J}(c_1,c_2,c_3,c_4,c_5)^T=(b_1,b_2,b_3,b_4,b_5)^T
\end{align}
where $b_i=\langle h\bpsi_j^\perp,K_i\bpsi_j^\perp\rangle$, $i=1,2,3,4,5.$
Since $\mathcal{J}$ is non-degenerate, by Lemma \ref{lem:K_i-prod} and \ref{lem:hK_i}, we get $c_1=O(1/j)$, $c_2=1/3+O(1/j)$, $c_3=O(1/j^2)$, $c_4=-1/3+O(1/j)$, and $c_5=O(1/j^2)$. We denote the solution of the above equations as $\vec{c}_*$.

Note that the Jacobian of $\Phi$ at $(0,\vec{c}_*)$ is \[(\partial_{c_1}\Phi,\partial_{c_2}\Phi,\partial_{c_3}\Phi,\partial_{c_4}\Phi,\partial_{c_5}\Phi)(0,\vec{c})=- \mathcal{J}.\]
Again the non-degeneracy of $\mathcal{J}$ implies that we can invoke the implicit function theorem. There exists $\varepsilon_j>0$ such that for any $0\leq \varepsilon<\varepsilon_j$, there exists $\vec{c}=\vec{c}(\varepsilon)=\vec{c}_*+O(\varepsilon)$ satisfies $\Phi(\varepsilon,\vec{c})=0$. That is,
\[\int_{\R}\v\cdot \DRh(K_i\bpsi_j^\perp)=0,\quad i=1,2,3,4,5.\]
Consequently $\u=\varepsilon \v+\bpsi_j$ also satisfies the above orthogonality. Using the form of $\v$, one readily check $\u=\varepsilon h_\perp\bpsi_j^\perp+\sqrt{1-\varepsilon^2h_\perp^2}\bpsi_j$ takes the desired form.


Since $K_i\in\ker \mathcal{L}_{\bpsi_j}$, then
\begin{align}
    \int_{\R} h_\perp \mathcal{L}_{\bpsi_j} h_\perp=\int_{\R} h\mathcal{L}_{\bpsi_j}h+O(\varepsilon)=O(1/\log j)+O(\varepsilon),
\end{align}
where we used \eqref{int_hLh-decay}. To establish \eqref{<hp,hp>}, we just need to use the results from Lemma \ref{lem:K_i-prod}, Lemma \ref{lem:rhohQh} and Lemma \ref{lem:hK_i}.
\begin{align*}
    \langle h_\perp\bpsi_j^\perp,h_\perp\bpsi_j^\perp\rangle
    =&\ \langle
    h\bpsi_j^\perp-\sum_{i=1}^5c_iK_i\bpsi_j^\perp,h\bpsi_j^\perp-\sum_{i=1}^5c_iK_i\bpsi_j^\perp\rangle \\
    =&\ \langle h\bpsi_j^\perp,h\bpsi_j^\perp\rangle-2\sum_{i=1}^5c_i\langle h\bpsi_j^\perp,K_i\bpsi_j^\perp\rangle+\sum_{k,l=1}^5c_kc_lJ_{kl}\\
    =&\  4\pi-\frac{2}{3}\pi-\frac{2}{3}\pi+\frac{1}{3}\pi+\frac{1}{3}\pi+O(1/\log j)+O(\varepsilon)\\
    =&\ \frac{10}{3}\pi+O(1/\log j)+O(\varepsilon).
\end{align*}
\end{proof}
\begin{proposition}\label{prop:inf-achieve}
Fix any $j\geq 100$. Suppose that $h_\perp$ and $\u$ are obtained from Lemma \ref{lem:hp-est}.
Then there exists $\varepsilon_j$ such that for $\varepsilon<\varepsilon_j$
the following infimum is achieved at $\bpsi_j$.
\begin{align}
    \inf_{\vartheta\in \S,\vec{\alpha}\in \mathbb{H}^2}\|\u-\bpsi_{\vartheta,\vec{\alpha}}\|^2_{\dot H^{1/2}(\R)}=\|\u-\bpsi_j\|^2_{\dot H^{1/2}(\R)}=\varepsilon^2\langle h_\perp,h_\perp\rangle+O(\varepsilon^3).
\end{align}
\end{proposition}
\begin{proof}
Since degree is continuous in $\dot H^{1/2}$-topology, there exists $\varepsilon'_j$ such that if $\varepsilon<\varepsilon'_j$, then $\|\u-\bpsi_j\|_{\dot H^{1/2}(\R)}\ll 1$ and $\deg \u=\deg\bpsi_j=2$. First, we claim the infimum is achieved. Indeed, take a minimizing sequence, $\vartheta_k$, $\vec{\alpha}_k=(\alpha_{1,k},\alpha_{2,k})$ such that
\begin{align}
    \|\u-\bpsi_{\vartheta_k,\vec{\alpha}_k}\|_{\dot H^{1/2}(\R)}^2\to\inf_{\vartheta\in \S,\vec{\alpha}\in \mathbb{H}^2}\|\u-\bpsi_{\vartheta,\vec{\alpha}}\|^2_{\dot H^{1/2}(\R)},\quad k\to \infty.
\end{align}
Since $\|\bpsi_{\vartheta_k,\vec{\alpha}_k}\|_{\dot H^{1/2}(\R)}^2=2$, going to a subsequence necessary, $\bpsi_{\vartheta_k,\vec{\alpha}_k}$ converges weakly to $\bpsi_*$. Due to the specific form of $\bpsi_{\vartheta_j,\vec{\alpha}_j}$, $\bpsi_*$ can takes three possible forms $\bpsi_{\vartheta_*, (\alpha_{1}^*,\alpha_{2}^*)}$, $\bpsi_{\vartheta_*,\alpha_*}$, and $e^{\im \vartheta_*}$.  If $\bpsi_{*}=\bpsi_{\vartheta_*,\alpha_*}$ or $e^{\im \theta_*}$, then $\|\bpsi_*\|_{\dot H^{1/2}(\R)}^2\leq 1$.
\begin{align}
    \|\u-\bpsi_{*}\|_{\dot H^{1/2}(\R)}\leq \liminf_{k\to \infty}\|\u-\bpsi_{\vartheta_k,\vec{\alpha}_k}\|_{\dot H^{1/2}(\R)}\leq \|\u-\bpsi_j\|_{\dot H^{1/2}(\R)}\ll 1.
\end{align}
On the other hand, by the Young's inequality and \eqref{R-1/2-norm}, we obtain
\begin{align}
    \|\u\|^2_{\dot H^{1/2}(\R)}\leq \frac{3}{2}\|\bpsi_*\|^2_{\dot H^{1/2}(\R)}+\frac{8}{3}\|\u-\bpsi_*\|^2_{\dot H^{1/2}(\R)}<2.
\end{align}
However, this contradict to the fact that $\deg \u=2$ and Theorem \ref{intro:lower-bd}.
Therefore we must have $\bpsi_*=\bpsi_{\vartheta_*, (\alpha_{1}^*,\alpha_{2}^*)}$, then $\|\bpsi_*\|_{\dot H^{1/2}(\R)}^2=2$, consequently $\bpsi_{\vartheta_k,\vec{\alpha}_k}$ converges to $\bpsi_*$ strongly and $\bpsi_*$ is one minimizer.

Suppose $\delta_j$ is defined in Proposition \ref{prop:module}. Apparently, there exists $\varepsilon_j>0$, such that for $\varepsilon<\varepsilon_j$, one has $\|\u-\bpsi_j\|_{\dot H^{1/2}(\R)}<\frac12\delta_j$.
Then any minimizer $\bpsi$ of the infimum satisfies
\begin{align}
\begin{split}
    \|\bpsi-\bpsi_j\|^2_{\dot H^{1/2}(\R)}&\leq 2\|\bpsi-\u\|^2_{\dot H^{1/2}(\R)}+2\|\u-\bpsi_j\|^2_{\dot H^{1/2}(\R)}\\
    &\leq4 \|\u-\bpsi_j\|^2_{\dot H^{1/2}(\R)}<\delta_j^2.
\end{split}
\end{align}
Suppose $u$ can be decomposed to
\begin{align}\label{u-decomp}
    \u=f \bpsi^\perp+\sqrt{1-f^2}\bpsi.
\end{align}
Since the infimum achieves at $\bpsi$, then
\begin{align}\label{critical-E-L}
    \int_{\R}\u\cdot \DRh (g \bpsi^\perp)=0,\quad \forall\, g\in \ker \mathcal{L}_{\bpsi}.
\end{align}
It follows from Proposition \ref{prop:module} that for any $u$ with $\|\u-\bpsi_j\|_{\dot H^{1/2}(\R)}<\frac{1}{2}\delta_j$, there exist unique $\vartheta,\vec{\alpha}$ such that  $\u$ satisfies \eqref{u-decomp} with $\bpsi=\bpsi_{\vartheta,\vec{\alpha}}$. This implies the minimizer is unique. Recall that the choice of $f=\varepsilon h_\perp$ with $\vartheta=0$ and $\vec{\alpha}=0$ make \eqref{u-decomp} and \eqref{critical-E-L} happen at the same time. Thus the infimum is achieved at $\bpsi_j$.

Finally, we can compute explicitly
\begin{align}
    \|\u-\bpsi_j\|_{\dot H^{1/2}(\R)}^2=\|\varepsilon h_\perp\bpsi_j^\perp+O(\varepsilon^2h_\perp^2)\bpsi_j\|_{\dot H^{1/2}(\R)}^2=\varepsilon^2\langle h_\perp,h_\perp\rangle +O(\varepsilon^3).
\end{align}




\end{proof}

Finally, we can prove the main theorem of this section.

\begin{proof}[{\bf Proof of Theorem \ref{intro:thm:un-stable}}]
We take $\u=\varepsilon h_\perp \bpsi_j^\perp+\sqrt{1-\varepsilon^2h_\perp^2}\bpsi_j$ as stated in Lemma \ref{lem:hp-est}.
Proposition \ref{prop:inf-achieve} implies that, if $\varepsilon<\varepsilon_j$, then
\begin{align}
    \inf_{\vartheta\in \S,\vec{\alpha}\in \mathbb{H}^2}\|\u-\bpsi_{\vartheta,\vec{\alpha}}\|^2_{\dot H^{1/2}(\R)}=\|\u-\bpsi_j\|^2_{\dot H^{1/2}(\R)}=\varepsilon^2\langle h_\perp,h_\perp\rangle+O(\varepsilon^3).
\end{align}
Using \eqref{<hp,hp>}, we obtain
\begin{align}\label{inf-LHS}
    \inf_{\vartheta\in \S,\vec{\alpha}\in \mathbb{H}^2}\|\u-\bpsi_{\vartheta,\vec{\alpha}}\|^2_{\dot H^{1/2}(\R)}=\frac{10}{3}\pi\varepsilon^2+O(\varepsilon^2/\log j+\varepsilon^3).
\end{align}
On the other hand, Lemma \ref{lem:vari-E} infers
\begin{align}
    \mathcal{E}(\u)=\mathcal{E}(\bpsi_j)+\varepsilon^2\int_{\R}h_\perp\mathcal{L}_{\bpsi_j}[h_\perp]+O(\varepsilon^3).
\end{align}
Note \eqref{1.1} implies $\mathcal{E}(\u)=\|\u\|^2_{\dot H^{1/2}(\R)}$ and  $\mathcal{E}(\bpsi_j)=\|\bpsi_j\|^2_{\dot H^{1/2}(\R)}=\deg\u=2$. Combing with \eqref{hpLhp}, it leads to
\begin{align}\label{deg-RHS}
    \|\u\|^2_{\dot H^{1/2}(\R)}-2=\varepsilon^2O(1/\log j).
\end{align}
Now compare \eqref{inf-LHS} and \eqref{deg-RHS} to get
\begin{align}
    \inf_{\vartheta\in \S,\vec{\alpha}\in \mathbb{H}^2}\|\u-\bpsi_{\vartheta,\vec{\alpha}}\|^2_{\dot H^{1/2}(\R)}\geq C(\log j)\left(\|u\|^2_{\dot H^{1/2}(\R)}-2\right).
\end{align}
Choosing $j$ sufficiently large such that $C\log j\geq M$, our theorem is established.
\end{proof}

\section*{Acknowledgement} The research of  B. Deng is supported by China Scholar Council and Natural Science Foundation of China (No. 11721101). The research of L. Sun and J. Wei is partially supported by NSERC of Canada.
\appendix
\section{ Proofs of Lemma \ref{lem:rhohQh} and Lemma \ref{lem:hK_i}}
In this appendix, we give the proofs of Lemma \ref{lem:rhohQh} and Lemma \ref{lem:hK_i}.
\begin{proof}[{\bf Proof of Lemma \ref{lem:rhohQh}}]
Notice that if $|x-j^2|\geq j$ and $x>0$, then
\begin{align}
    \rho_j(x)\leq\frac{C}{1+(x-j^2)^2} .
\end{align}
Since $\rho_j$ is even, then we have
\begin{align}
    \begin{split}
        \int_{\R} \rho_j h^2
    =&\int_{0}^\infty\frac{8(1+j^4+x^2)h(x)^2}{[1+(x-j^2)^2][1+(x+j^2)^2]}dx\\
    =&\int_{j^2-j}^{j^2+j}\frac{8(1+j^4+x^2)}{[1+(x-j^2)^2][1+(x+j^2)^2]}dx+O(1/j)\\
    =&\int_0^\infty\frac{8(1+j^4+x^2)}{[1+(x-j^2)^2][1+(x+j^2)^2]}dx+O(1/j)
    =\ 4\pi+O(1/j).
    \end{split}
\end{align}
Then we prove \eqref{int_rhojh}.
To prove \eqref{int_Qjh}, we
divide the support of $h(x)h(y)$, i.e., $\{|x|\leq2j^2\}\cap\{|y|\leq2j^2\}$ into $\Omega_i$, $i=1,\cdots,9$, according to their types.
Let
\begin{align}\nonumber
    \begin{split}
        \Omega_1=\{|x-j^2|\leq j, j\leq|y-j^2|\leq j^2\}\cup\{j\leq |x-j^2|\leq j^2,|y-j^2|\leq j\},\\
        \Omega_2= \{|x-j^2|\leq j, j\leq|y+j^2|\leq j^2\}\cup\{j\leq |x-j^2|\leq j^2,|y+j^2|\leq j\},\\
        \Omega_3=\{|x+j^2|\leq j, j\leq|y+j^2|\leq j^2\}\cup\{j\leq |x+j^2|\leq j^2,|y+j^2|\leq j\},\\
        \Omega_4=\{|x+j^2|\leq j, j\leq|y-j^2|\leq j^2\}\cup\{j\leq |x+j^2|\leq j^2,|y-j^2|\leq j\}.
    \end{split}
\end{align}
Using the expression \eqref{R-ker}, and there $ |1+xy+j^4|\leq 8 j^4$, it is easy to see that
\begin{align}\nonumber
   Q_j(x,y)\leq C \begin{cases}
    [1+(x-j^2)^2]^{-1}[1+(y-j^2)^2]^{-1},\quad  \Omega_1,\\
    [1+(x-j^2)^2]^{-1}[1+(y+j^2)^2]^{-1},\quad  \Omega_2,\\
    [1+(x+j^2)^2]^{-1}[1+(y+j^2)^2]^{-1},\quad  \Omega_3,\\
    [1+(x+j^2)^2]^{-1}[1+(y-j^2)^2]^{-1},\quad  \Omega_4.
    \end{cases}
\end{align}
Consequently, for example, on  $\Omega_1$ we have, by the symmetry of $x$ and $y$,
\begin{align}\label{int-omega1}
\begin{split}
     \left|\iint_{\Omega_1} Q_j(x,y)h(x)h(y)dxdy\right|&\leq C \iint_{\{|x-j^2|\leq j, j\leq |y-j^2|\leq j^2\}} \frac{dxdy}{[1+(x-j^2)^2][1+(y-j^2)^2]}\\
     &=4C\arctan j\int_j^{j^2}\frac{dy}{  1+y^2}=O(1/j).
\end{split}
\end{align}
Similarly,  the integral on  $\Omega_i, i=2,3,4$ is also of order $O(1/j)$.

On both of sets $\Omega_5=\{|x-j^2|\leq j\}\cap \{|y+j^2|\leq j\}$ and $\Omega_6=\{|x+j^2|\leq j\}\cap \{|y-j^2|\leq j\}$, it holds that
\begin{align}
    |1+xy+j^4|\leq 4j^3.
\end{align}
It is easy to see that
\begin{align}\nonumber
   Q_j(x,y)\leq \frac{C}{j^2}\begin{dcases}
    [1+(x-j^2)^2]^{-1}[1+(y+j^2)^2]^{-1},\quad  \Omega_5,\\
   [1+(x+j^2)^2]^{-1}[1+(y-j^2)^2]^{-1},\quad  \Omega_6.
    \end{dcases}
\end{align}
Then we obtain
\begin{align}\label{int-omega7}
    \left|\iint_{\Omega_i}Q_j(x,y)h(x)h(y)dxdy\right|=O(1/j^2),\quad i=5,6.
\end{align}
Let $\Omega_7=\{|x-j^2|\leq j\}\cap \{|y-j^2|\leq j\}$ and $\Omega_8=\{|x+j^2|\leq j\}\cap \{|y+j^2|\leq j\}$, and $\Omega_9=\{|x|\leq 2j^2,|y|\leq 2j^2\}\setminus\cup_{i=1}^8 \Omega_i$. On $\Omega_9$, we have
\begin{align}\label{int-omega5}
    \begin{split}
           \left|\iint_{\Omega_9} Q_j(x,y)h(x)h(y)dxdy\right| &\leq C\int_j^{j^2}\int_j^{j^2}\frac{dxdy}{(1+x^2)(1+y^2)}=O(1/j^2).
    \end{split}
\end{align}
Together with \eqref{int-omega1}, \eqref{int-omega5} and \eqref{int-omega7},  we have
\begin{align}
    \begin{split}
          \frac{1}{2\pi}&\iint_{\R\times \R}Q_j(x,y)h(x)h(y)dxdy\\
     =&\ \frac{1}{2\pi}\iint_{\Omega_7} Q_j(x,y)dxdy  +\frac{1}{2\pi}\iint_{\Omega_8} Q_j(x,y)dxdy
     +O(1/j)\\
    =&\ \frac{1}{2\pi}\iint_{\R\times \R} Q_j(x,y)dxdy+O(1/j)= 4\pi+O(1/j).
    \end{split}
\end{align}
This is \eqref{int_Qjh}. Once we obtain \eqref{int_rhojh},  \eqref{int_Qjh} and Lemma \ref{lem:1/2norm-fj}, it follows that
\begin{align}
    \begin{split}
        \int_{\R}h\mathcal{L}_{\bpsi_j}h
        =&\ \int_{\R}h\DRh h-\int_{\R} \rho_j h^2+\frac{1}{2\pi} \iint_{\R\times \R} Q_j(x,y)h(x)h(y)dxdy\\
        =&\ O(1/\log j)
    \end{split}
\end{align}
and
\begin{align}
    \begin{split}
      \langle h\bpsi_j^\perp,h\bpsi_j^\perp\rangle&=\int_{\R}h\DRh h+\frac{1}{2\pi} \iint_{\R\times \R} Q_j(x,y)h(x)h(y)dxdy\\
      &=4\pi+O(1/\log j).
    \end{split}
\end{align}
\end{proof}


\begin{proof}[{\bf Proof of Lemma \ref{lem:hK_i}}]
Since $h$ is odd and $Q(x,y)=Q(y,x)=Q(-x,-y)$, it is easy to get
\begin{align}
      \langle h\bpsi_j^\perp,K_1\bpsi_j^\perp\rangle&=\frac{1}{2\pi}\iint_{\R\times\R}Q_j(x,y)h(x)dxdy=0.
\end{align}
By a direct computation, we have
\begin{align}\nonumber
    \begin{split}
        \langle h\bpsi_j^\perp,K_2\bpsi_j^\perp\rangle&=\int_{\R}h\DRh K_2+\frac{1}{2\pi}\iint_{\R\times\R}Q_j(x,y)h(x)K_2(y)dxdy\\
        &=\int_{\R}\frac{2(1-(x-j^2)^2)h(x)}{[1+(x-j^2)^2]^2}dx+\int_{\R}\frac{2(3+(x+j^2)^2)h(x)}{[1+(x-j^2)^2][1+(x+j^2)^2]}dx\\
        &=I_1+I_2.
    \end{split}
\end{align}
Since
\begin{align}\label{hK2-int1}
    \begin{split}
    I_1=\int_{|x|\leq j}\frac{2(1-x^2)h(x)}{(1+x^2)^2}&\leq  \left|\int_{|x|\leq j}\frac{2(1-x^2)}{(1+x^2)^2}dx\right|+\int_{|x|\geq j}\left|\frac{2(1-x^2)}{(1+x^2)^2}\right|dx\\
    &= O(1/j),
    \end{split}
\end{align}
and
\begin{align}\label{hK2-int2}
    \begin{split}
    I_2 &= \int_{\R}\frac{2h(x)}{1+(x-j^2)^2}dx+\int_{\R}\frac{4h(x)}{[1+(x-j^2)^2][1+(x+j^2)^2]}dx\\
    &= 2\arctan j+O(1/j^2)= \pi+O(1/j).
    \end{split}
\end{align}
Then we get $\langle h\bpsi_j^\perp,K_2\bpsi_j^\perp\rangle=\pi+O(1/j)$.

We also have
\begin{align}\nonumber
    \begin{split}
      \langle h\bpsi_j^\perp,K_3\bpsi_j^\perp\rangle&=\int_{\R}\frac{4(x-j^2)h(x)}{[1+(x-j^2)^2]^2}dx+\int_{\R}\frac{4(x-j^2)h(x)}{[1+(x-j^2)^2][1+(x+j^2)^2]}dx.
    \end{split}
\end{align}
Since
\begin{align}
    \begin{split}
        \left|\int_{\R}\frac{4(x-j^2)h(x)}{[1+(x-j^2)^2]^2}dx\right|\leq\int_{|x|\geq j}\frac{4x}{(1+x^2)^2}dx=O(1/j^2),
    \end{split}
\end{align}
and
\begin{align}
    \begin{split}
       &\left|\int_{\R}\frac{4(x-j^2)h(x)}{[1+(x-j^2)^2][1+(x+j^2)^2]}dx\right|\\
       &\leq  \left|\int_{|x-j^2|\leq j}\frac{4(x-j^2)}{j^4[1+(x-j^2)^2]}dx\right|+ \left|\int_{|x-j^2|\geq j}\frac{4}{j[1+(x+j^2)^2]}dx\right|= O(1/j^2).
    \end{split}
\end{align}
Then we get $\langle h\bpsi_j^\perp,K_3\bpsi_j^\perp\rangle=O(1/j^2)$.
Similarly, note that $h(x)=-1$ when $|x+j^2|\leq j$, we can get
\begin{align}
    \langle h\bpsi_j^\perp,K_4\bpsi_j^\perp\rangle=-\pi+O(1/j),\quad \langle h\bpsi_j^\perp,K_5\bpsi_j^\perp\rangle=O(1/j^2).
\end{align}
\end{proof}
\small
\bibliographystyle{plainnat}
\bibliography{Half-HarmonicMap-ref}

\end{document}